\documentclass[10pt,a4paper,reqno]{amsart}
\usepackage{amsthm,amssymb,amstext,graphicx, comment}
\usepackage{bbm}
\usepackage{enumerate}
\usepackage{amscd}
\usepackage{amssymb}
\usepackage{amsthm}
\usepackage{mathtools}
\usepackage{enumitem}
\usepackage[utf8]{inputenc}
\usepackage[T1]{fontenc}
\usepackage{lmodern}
\usepackage{bibentry}
\usepackage{soul}
\usepackage{hyperref}
\hypersetup{
  colorlinks=true,
  pdfpagelabels=true,
  pdfpagelayout==TwoColumnRight,
  pdftitle={},
  pdfauthor={© Sven Karbach},
  pdfsubject={Project1},
  pdfkeywords={}
}

\DeclareMathOperator\dom{dom}

\DeclareMathOperator\lin{lin}

\DeclareMathOperator*{\esssup}{ess\,sup}
\newcommand*\D{\mathop{}\!\textnormal{d}}

\newcommand*\E{\mathop{}\!\textnormal{e}}

\numberwithin{equation}{section}

\newtheorem{theorem}{Theorem}[section]
\newtheorem{lemma}[theorem]{Lemma}
\newtheorem{proposition}[theorem]{Proposition}
\newtheorem{corollary}[theorem]{Corollary}
\theoremstyle{plain}

\theoremstyle{definition}
\newtheorem{definition}[theorem]{Definition}
\newtheorem{remark}[theorem]{Remark}

\newtheoremstyle{example}
  {.3\baselineskip}
  {.3\baselineskip}
  {\normalsize}  
  {0pt}       
  {\bfseries} 
  {.}         
  {5pt plus 1pt minus 1pt} 
  {}          
\theoremstyle{example}

\newtheorem*{assumption*}{\assumptionnumber}
\providecommand{\assumptionnumber}{}
\makeatletter

\newenvironment{proofoutline}
{\proof[Outline of the proof]}
{\endproof}
\makeatother
\newtheorem{example}[theorem]{Example}
\setlist[enumerate,1]{label=\roman*),ref=\roman*)}

\def\e{\operatorname{e}} 
\def\eps{\varepsilon}

\renewcommand{\MR}{\mathbb{R}}

\newcommand{\MN}{\mathbb{N}}
\newcommand{\MP}{\mathbb{P}}
\newcommand{\MQ}{\mathbb{Q}}

\newcommand{\R}{\MR}
\newcommand{\N}{\MN}

\newcommand{\cF}{\mathcal{F}}
\newcommand{\cA}{\mathcal{A}}
\newcommand{\cB}{\mathcal{B}}

\newcommand{\cD}{\mathcal{D}}

\newcommand{\cH}{\mathcal{H}}

\newcommand{\cL}{\mathcal{L}}

\newcommand{\cS}{\mathcal{S}}
\newcommand{\cG}{\mathcal{G}}

\newcommand{\sI}{\mathsf{I}}

\newcommand{\df}{\coloneqq}

\newcommand{\one}{\mathbf{1}}
\newcommand{\interior}[1]{({\kern0pt#1})^{\textnormal{o}}}
\newcommand{\set}[1]{\left\{ #1\right\}}
\newcommand{\norm}[1]{\|#1\|}
\newcommand{\llangle}{\langle\langle}
\newcommand{\rrangle}{\rangle\rangle}

\newcommand{\EX}[1]{\mathbb{E}\left[#1\right]}
\newcommand{\EXspec}[2]{\mathbb{E}_{#1}\left[#2\right]}

\newcounter{Task}

\newcommand{\cHplus}{\cH^{+}}

\newcommand{\cHpluso}{\cHplus\setminus \{0\}}
\newcommand{\MRplus}{\MR^{+}}
\newcommand{\dm}{m(\D\xi)}
\newcommand{\dmu}{\mu(\D\xi)}
\newcommand{\dmk}{m^{(k)}(\D\xi)}
\newcommand{\dmuk}{\mu^{(k)}(\D\xi)}

\newcommand{\sgc}{\begin{color}{blue}}
\newcommand{\cgs}{\end{color}}

\newcommand{\sgcrep}{\begin{color}{red}}
\newcommand{\cgsrep}{\end{color}}
\begin{document}
\title[]{Affine pure-jump processes on positive Hilbert-Schmidt operators}  
\author{Sonja Cox, Sven Karbach, Asma Khedher}
\thanks{The authors gratefully acknowledge Christa Cuchiero for fruitful
  discussions. Moreover, this research is partially funded by The Dutch Research
  Council (NWO)}
\begin{abstract}
We show the existence of a broad class of affine Markov processes on the cone
of positive self-adjoint Hilbert-Schmidt operators. Such processes are
well-suited as infinite-dimensional stochastic covariance
models. The class of processes we consider is an infinite-dimensional analogue
of the affine processes on the cone of positive semi-definite and symmetric matrices studied in Cuchiero et al.\ [\textit{Ann.\ Appl.\ Probab.} \textbf{21} (2011) 397--463]. 

As in the finite-dimensional case, the processes we construct allow for a
drift depending affine linearly on the state, as well as jumps governed by a jump
measure that depends affine linearly on the state. 
The fact that the cone of positive self-adjoint Hilbert-Schmidt operators
has empty interior calls for a new approach to proving existence: instead of using standard
localisation techniques, we employ the theory on generalized Feller semigroups
introduced in D\"orsek and Teichmann [\textit{arXiv} (2010)] and further
developed in Cuchiero and Teichmann [\textit{J.\  Evol.\ Equ.} \textbf{20} (2020) 1301--1348]. Our approach requires a second moment condition on the jump measures involved, consequently, we obtain explicit formulas for the first and second moments of the affine process.
\end{abstract}
\maketitle

\section{Introduction}
In this article we show the existence of time-homogeneous {\it affine} Markov processes on the {\it cone of
positive self-adjoint Hilbert-Schmidt} operators. The affine class is known for its tractability and flexibility. 

It is tractable because the Fourier-Laplace transform of such processes depends in an exponentially affine way on the initial state vector of the process. More specifically,
denote by $(\mathcal{H}, \langle \cdot, \cdot\rangle)$ the Hilbert space of
self-adjoint Hilbert-Schmidt operators on a Hilbert space
$(H,\langle\cdot,\cdot\rangle_{H})$ and by $\mathcal{H}^+\subseteq \cH$ the
cone of positive self-adjoint Hilbert-Schmidt operators. A $\cHplus$-valued
time-homogeneous Markov process $(X_{t})_{t\geq 0}$ is affine, if there exist
functions $\phi\colon\MRplus\times \mathcal{H}^+ \to \MRplus$, $\psi\colon
\MRplus\times \mathcal{H}^+\to \mathcal{H}^+$ such that
\begin{align}\label{eq:affine-formula-intro}
 \EX{\E^{-\langle X_{t},u\rangle}\rvert X_{0}=x}=\E^{-\phi(t,u)-\langle x,
  \psi(t,u)\rangle}, \qquad t\geq 0\,,
  \end{align}
for all $u\in\mathcal{H}^+$. The functions $\phi$ and $\psi$
are typically solutions of ordinary differential equations given in terms of the parameters of the model. 

The affine class is 
flexible because the parameters of the model satisfy certain assumptions that allow for desired features such as constant and
bounded linear drifts and constant and affine state-dependent jumps of infinite-variation. 

Our motivation for studying affine processes in the state space $\cHplus$ lies
in the fact that such processes are well-qualified as models for infinite
dimensional covariance processes, i.e., they can be used for the modeling of
stochastic volatility in, for example, bond and commodity markets. See
e.g.~\cite{Fil01, CT06, BK14, BK15} for the modeling of forward price dynamics
in bond and commodity markets as a process with values in a Hilbert
space. In particular, in \cite{BRS15} a stochastic volatility model is
constructed that involves a covariance process driven by L\'evy noise and
taking values in the positive Hilbert-Schmidt operators. Our model extends the
covariance model in~\cite{BRS15} from L\'evy driven processes to processes
allowing for state-dependent jumps (see also~\cite[Section
4.1]{cox2021infinitedimensional}). More specifically, the affine processes we
consider in this paper are of {\it pure-jump} type
where the jumps can be {\it state-dependent} and of {\it infinite variation}.

Let us state our main result in an abbreviated form, see also
Theorem~\ref{thm:existence-affine-process} below and its proof:
\begin{theorem}\label{thm:existence-affine-process-intro}
Let $(b,B,m,\mu)$ be a tuple consisting of a vector $b\in \mathcal{H}$, a bounded linear operator $B\in
\cL(\mathcal{H})$, a measure  $m$ on the Borel-$\sigma$-algebra $\cB(\mathcal{H}^+\setminus
\set{0})$ and a $\mathcal{H}$-valued measure $\mu$ on $\cB(\mathcal{H}^+\setminus \set{0})$,
satisfying the \emph{admissibility assumptions} posed in Definition
\ref{def:admissibility} below. Then there exists an affine process $(X_{t})_{t\geq 0}$ in $\mathcal{H}^+$, such that the functions $\phi$ and
$\psi$ in equation \eqref{eq:affine-formula-intro} are the 
unique solution to the so called \emph{generalized Riccati equations}
associated to $(b,B,m,\mu)$:
\begin{align}
 \frac{\partial}{\partial  t}\phi(t,u)&=\langle b,\psi(t,u)\rangle-\int_{\cHpluso}\big(\E^{-\langle
  \xi,\psi(t,u)\rangle}-1+\langle \chi(\xi),\psi(t,u)\rangle\big)\dm,\label{eq:Riccati-intro-phi} \\
  \frac{\partial}{\partial t}\psi(t,u)&=B^{*}(\psi(t,u))-\int_{\cHpluso}\big(\E^{-\langle
  \xi,\psi(t,u)\rangle}-1+\langle \chi(\xi),
  \psi(t,u)\rangle\big)\frac{\dmu}{\norm{\xi}^{2}}\,,\label{eq:Riccati-intro-psi} 
\end{align}
with initial values $\phi(0,u)=0$ and $\psi(0,u)=u$ for $u\in \mathcal{H}^+$.  
\end{theorem}
More specifically, the processes we consider have a constant drift vector $b$, 
a linear drift term $B$, a constant jump measure $m$, and a state-dependent jump measure $\mu$. In addition to Theorem~\ref{thm:existence-affine-process-intro}, and as a by-product of our method of proof, we establish explicit formulas for the first and second moments of the affine processes, see Proposition~\ref{prop:explicit-formula}. \par 

Note that equation \eqref{eq:Riccati-intro-psi} is a non-linear
differential equation on the cone of positive self-adjoint Hilbert-Schmidt
operators which, in general, cannot be solved explicitly. Numerical methods for approximating solutions to infinite-dimensional Riccati equations are considered in e.g.~\cite{EEM19} and~\cite{Ros91}. A numerical approximation method tailored for this specific equation will be analysed in forthcoming work~\cite{Kar21}.

There is a vast number of articles dealing with affine processes in several
state spaces in {\it finite dimensions}, we mention, for example,
\cite{cuchiero2011affine, DFS03, KST13, KM15, SVV11, KM10, CFMT11}. 
In \cite{DFS03} and \cite{CFMT11}, the authors considered affine processes respectively on the canonical state space
$\MR_{+}^{d}\times\MR^{m}$, $d,m\in\MN$, and on the cone of positive semi-definite symmetric matrices. Both articles give
sufficient and necessary admissible parameter conditions and characterize the
class of stochastically continuous affine processes by means of their
Markovian generator. 
The literature on affine processes in {\it infinite-dimensional} state spaces is
more sparse. Existence of affine diffusion processes
on Hilbert spaces was investigated in \cite{STY20}. In \cite{Gra16}, the
author investigated affine processes in general locally convex vector spaces
and in \cite{CT20}, existence of affine Markovian lifts of finite-dimensional Volterra processes
was shown. The Markovian lift process takes values in a certain cone in a
space of measures and shares many features of the affine processes which we
consider. 


The biggest challenge we face is that like many infinite-dimensional cones, the cone of positive
self-adjoint Hilbert-Schmidt operators has empty interior. One consequence is that one cannot employ classical localisation arguments to establish existence of the desired processes; we take a different approach outlined below. Another consequence is that it is difficult to incorporate a diffusion term. Indeed, although formally this involves a non-commutative version of the superprocesses studied in e.g.~\cite{KS88}, the methods in~\cite{KS88} break down in the non-commutative setting. Thus it remains an open question whether and under what conditions infinite-dimensional affine processes on positive Hilbert-Schmidt operators allow for a diffusion term. 


Our new approach involves approximating the transition semigroup associated with our Markov process by simpler 
transition semigroups corresponding to affine finite-activity jump
processes. We then exploit the {\it generalized Feller theory} introduced in
\cite{DT10} and the approximation results \cite[Proposition 3.3 and Theorem
3.2]{CT20} as well as a version of the Kolmogorov extension theorem proven in
\cite[Theorem 2.11]{CT20} to show that the limiting semigroup gives rise to a
generalized Feller process. Note that the idea of showing the existence of affine processes with
jumps of infinite variation through
an approximation with simpler affine processes was already used on e.g.\ convex sets
in finite dimensions, where it is known that affine processes are (classical)
Feller processes (see \cite{DFS03} and \cite{CFMT11}). However, our approach is
somewhat different, and a considerable amount of effort goes into verifying
that the approximating generalized Feller semigroups satisfy all necessary
conditions to ensure convergence. In particular, a subtle analysis of the
regularity of $\phi$ and $\psi$ is conducted and we derive a uniform growth
bound for the approximating semigroups.

\subsection{Layout of the article}
In Section \ref{sec:setting-main-result} we provide the definition of admissible parameter sets and we
state our main result (Theorem~\ref{thm:existence-affine-process}) on the
existence of affine pure-jump processes on the cone of positive self-adjoint
Hilbert-Schmidt operators. Moreover, we specify the exact form of the weak
generator of these Markov affine processes on the linear span of the Fourier
basis elements in terms of the introduced admissible parameter set. A brief outline of the proof of 
Theorem~\ref{thm:existence-affine-process} is presented in Section~\ref{sec:setting-main-result} and the full proof is left to Section~\ref{sec:exist-affine-pure}. 
In Section~\ref{sec:an-associated-class} we show the existence and uniqueness of the
solution to the generalized Riccati equations~\eqref{eq:Riccati-intro-phi} and~\eqref{eq:Riccati-intro-psi} and we study the
regularity of this solution with respect to its initial value. We recall the
generalized Feller setting in Section~\ref{ssec:prel-gener-fell}. Then in
Section~\ref{ssec:gener-fell-semigr} and~\ref{proof:existence-affine-process}
we making use of the results in Section~\ref{sec:an-associated-class} and
some intricate approximation techniques for generalized Feller semigroups to
complete the proof of Theorem~\ref{thm:existence-affine-process}. In Appendices~\ref{sec:comparison-theorem}, \ref{sec:integr-with-resp} and \ref{sec:proof-lemma-refl}, we, respectively, add a comparison theorem that
we need in our derivations, collect some 'standard' results on integration
with respect to vector-valued measures, and provide a regularity result of the solution to our considered generalized Riccati equations.
\subsection{Notation}\label{ssec:notation}
We set $\MN= \{1,2,\ldots\}$ and $\MN_0 = \{0,1,\ldots\}$. 
For a vector space $X$
and $U\subseteq X$ we denote the linear span of $U$ by $\lin(U)$.
For $(X,\tau)$ a topological space and $S\subseteq X$ we let $\cB(S)$
denote the \emph{Borel-$\sigma$-algebra} generated by the relative topology on $S$. 
Let $(H,\langle \cdot,\cdot \rangle_H)$ be a Hilbert space. Then we denote by $C(S,H)$ 
the space of $H$-valued functions on $S$ that 
are continuous with respect to the relative topology
and we denote by $C_b(S,H)$ the space of bounded $H$-valued continuous functions on $S$.
 This is a Banach space 
when endowed with the supremum norm $\| \cdot \|_{C(S)}$. 
Notice that when $H=\R$, we typically omit $H$ in the notation: $C(S):=C(S,\R)$.
Let $\cL(X)$ denote the space of \emph{bounded linear operators} from a Banach space $X$ to $X$.
This
is a Banach space when equipped with the operator norm $\| \cdot
\|_{\cL(X)}$.
If $\cG$ is a linear operator on a Banach space $X$, we denote 
its \emph{domain} by $\dom(\cG)$ 
and denote by $\sI$ the identity in $\cL(X)$. We denote unbounded operators by a
calligraphic font and bounded ones by the standard font, e.g., $\cG$ versus $G$. 
Let $\cL^{(2)}(H\times H,H)$
denote the space of continuous bilinear forms from $H\times H$ to $H$.
The \emph{adjoint} of an operator $A \colon  H \rightarrow H$ is denoted by $A^{*}$. 
An operator $A\in \cL(H)$
is \emph{positive} if $\langle A x , x \rangle_H \geq 0$ for all $x\in H$.
We let $\cL_{2}(H)$ denote the
space of \emph{Hilbert-Schmidt operators} from $H$ to $H$, this is a Hilbert space when endowed with the inner product
\begin{align*}
 \langle A, B \rangle_{\cL_2(H)} = \sum_{n=1}^{\infty} \langle A e_n, B e_n \rangle_H,
\end{align*}
where $(e_n)_{n\in \MN}$ is an orthonormal basis for $H$ and $\langle \cdot, \cdot \rangle_{\cL_2(H)}$
is independent of the choice of the orthonormal basis (see, e.g.,~\cite[Section VI.6]{Wer00}). 
A nonempty subset $K$ of a vector space is called a \emph{wedge} if $K+K\subseteq K$ and $\alpha K \subseteq K$
for all $\alpha \geq 0$, if moreover $K\cap (-K) = \{ 0\}$ then we call $K$ a \emph{cone}. A cone $K$ in a vector space $X$ induces a partial ordering: we 
write $x \leq_{K} y$ if $y-x\in K$ (and $x\geq_K y$ if $x-y\in K$). If 
$K\subset H$ is a wedge, 
we define the \emph{dual} of $K$ by
\begin{equation}
K^* = \{ x \in H \colon \langle x, y \rangle_H \geq 0 \text{ for all } y\in K \},
\end{equation}
and we say that $K$ is \emph{self-dual} if $K=K^*$. Note that if $K$ is self-dual then $0 \leq_K x\leq_K y$ implies 
$\| x \|_H^2 \leq \langle x , y \rangle_H  \leq \| x \|_H \| y \|_H$, i.e.,
\begin{equation}\label{eq:monotonic}
 0\leq_K x \leq_K y \Rightarrow \| x \|_H \leq \| y \|_H
\end{equation}
(in other words, $K$ is \emph{monotonic}).

We say that a cone $K$ is \emph{regular} if for all $y, x_1,x_2, \ldots \in K$
satisfying $x_1\leq_K x_2 \leq_K \ldots \leq_K y$ there exists an $x\in H$
such that $\lim_{n\rightarrow \infty} \| x_n - x \|_H = 0$.
A cone $K$ is said to have \emph{generating dual} if $B^{*}=K^{*}-K^{*}$. It
is true that $K$ has generating dual if and only if $K$ is \emph{normal},
i.e. $0\leq_{K}x\leq_{K} y$ for $y\in K$, implies $\norm{x}\leq \lambda
\norm{y}$ where $\lambda>0$, see e.g. \cite{KG85}. In finite dimensions, self-dual
normal cones have non-empty interior. However, in infinite dimensions,
the property $H=K-K$ does in general not imply that $K$ has non-empty
interior, see \cite{KLS89}.
Let $(S,\cS)$ be a measurable space and $U\subseteq H$. A mapping $\mu\colon \cS \rightarrow U$ is called a  \emph{$U$-valued measure} (on $S$) if it is weakly countably additive, i.e., if 
for every pairwise disjoint sequence $U_1, U_2 \ldots \in \cS$ satisfying $\cup_{n\in \MN} U_n = U$ it holds that
\begin{align*}
\langle \mu(U)x, y \rangle_H = \sum_{k \in \mathbb{N}} \langle \mu(U_k)x, y\rangle_H
\end{align*}
for all $x,y \in H$. We know 
from the work of Pettis~\cite{Pet38} that if $\mu\colon \cF \rightarrow H$ 
is weakly $\sigma$-additive, then it is also strongly $\sigma$-additive.  
For a $H$-valued measure $\mu$ and $h\in H$ we define the signed measure 
$\langle \mu, h \rangle \colon \cF \rightarrow \R$ by $\langle \mu, h \rangle (A) 
= \langle \mu(A), h\rangle_{H}$, $A\in \cF$. Throughout this work we are required to integrate with
respect to vector-valued measures, for a better readability we added a section
on this matter to Appendix \ref{sec:integr-with-resp}.

\subsection{Setting}\label{ssec:setting}
Throughout this article we let $(H, \langle \cdot,\cdot\rangle_H)$ be a
separable infinite-dimensional real Hilbert space.
For notational brevity we reserve $\langle \cdot, \cdot \rangle$ to denote the inner product on $\cL_2(H)$,
and $\| \cdot \|$ for the norm induced by $\langle \cdot, \cdot \rangle$. 
In addition, we define $\cH$ to be the space of all self-adjoint Hilbert-Schmidt operators on $H$ and 
$\cH^+$ to be the cone of all positive operators in $\cH$:
\begin{equation*}
 \cH \df \{ A \in \cL_2(H) \colon A = A^* \}, \ \text{and}\ 
 \cH^{+} \df \{ A \in \cH \colon \langle Ah, h\rangle_H \geq 0 \text{ for all } h\in H \}. 
\end{equation*}
Note that $\cH$ is a closed subspace of $\cL_2(H)$, and that $\cH^{+}$ is a self-dual 
cone in $\cH$ (indeed, $(\cH^+)^*\subseteq \cH^+$ by the spectral theorem for compact operators, 
and the reverse inclusion is trivial). Consequently, $\cH$ is monotonic.
Moreover, $\cH^{+}$ is regular (see, e.g., \cite[Theorem 1]{Kar59}), we have
$\cH=\cHplus-\cHplus$ and $\cHplus$ has empty interior.

 
We define the truncation function $\chi: \mathcal{H} \rightarrow  \mathcal{H}$
by $\chi(\xi) = \xi\one_{\{\|\xi\|\leq 1\}}$ and fix it throughout this work.

\section{Affine processes on \texorpdfstring{$\cHplus$}{the cone of positive
    self-adjoint Hilbert-Schmidt operators} and statement of main result}\label{sec:setting-main-result}
In this section we give a detailed definition of affine processes on the state
space $\cHplus$ and introduce the notion of admissible parameter sets. We
compare our admissible parameter conditions with the matrix valued case, this
is done in Remark \ref{rem:comparison_finite_dim_admissability}. Given an admissible
parameter set we deduce first properties of the right-hand side functions of
the differential equations in \eqref{eq:Riccati-intro-phi}-\eqref{eq:Riccati-intro-psi}. At the end of this
section we state our main result of this article in Theorem
\ref{thm:existence-affine-process}, which guarantees the existence of affine Markov processes
on $\cHplus$ associated with a given admissible parameter set and specifies the form of
their weak generator on the Fourier-basis elements. However, we
postpone the proof to Section \ref{proof:existence-affine-process} and only
give a brief outline at the end of this section.\\
We consider a time-homogeneous Markov process $X$ with state space
$\mathcal{H}^+$ and transition semigroup $(P_t)_{t\geq 0}$ acting on functions
$f \in C_b(\mathcal{H}^+)$,
$$P_tf(x) = \int_{\mathcal{H}^+} f(\xi) p_t(x, \D \xi)\,, \qquad  x \in \mathcal{H}^+\,,$$
where $p_t(x,\cdot)$, $t \geq 0$, $x \in \mathcal{H}^+$, is the transition kernel of $X$. Moreover for $x\in\cHplus$, we denote the law of $X$ given $X_{0}=x$ by $\mathbb{P}_{x}$.
\begin{definition}
The Markov process $(X, (\mathbb{P}_x)_{x \in \mathcal{H}^+})$ is called {\it affine} if its Laplace transform has exponential-affine dependence on the initial state, i.e., if 
\begin{align}\label{eq:laplace-transform}
P_t \e^{-\langle x, u\rangle} = \int_{\mathcal{H}^+}\e^{-\langle u, \xi\rangle} p_t(x, \D \xi) = \e^{-\phi(t,u) -\langle x,\psi(t,u)\rangle}\,,
\end{align}
for all $t\geq 0$, and $u, x \in \mathcal{H}^+$, for some functions
$\phi\colon \mathbb{R}_+ \times \mathcal{H}^+ \rightarrow \MR_+$ and $\psi\colon
\MR_+ \times \mathcal{H}^+ \rightarrow \mathcal{H}^+$.
\end{definition}
We follow the approach in \cite{CFMT11} and consider the Laplace transform
instead of the characteristic function which is justified by the
non-negativity of $X$. \par
Note, that we do not require stochastic continuity of the affine process here,
as in this work we are not aiming to provide a characterization of
affine processes. As discussed in the introduction, our existence result
requires an analysis of the corresponding generalized Riccati equations. In particular, a direct consequence of our approach (see Theorem~\ref{thm:existence-affine-process} below) is that the processes we consider are \emph{regular} in the sense of~\cite[Def.\ 2.2]{CFMT11}. We recall this concept for the reader's convenience:
\begin{definition}\label{def:regular}
  We call the affine process \textit{regular}, whenever the functions
  \begin{align*}
   \frac{\partial \phi(t,u)}{\partial t}\rvert_{t=0+}\qquad\text{ and } \qquad \frac{\partial \psi(t,u)}{\partial t}\rvert_{t=0+}, 
  \end{align*}
  exist and are continuous at $u=0$.
\end{definition}
As we will see, the established class of affine processes satisfy an even
stronger regularity condition, see Section~\ref{ssec:regul-with-resp}. In
finite dimensions stochastically continuous affine processes are always
regular (see \cite{KST13}), however, there exist finite-dimensional affine
processes that are not stochastically continuous. Arguably, such processes are
of minor interest in applications. In infinite dimensions the regularity
condition is somewhat more restrictive, as it implies e.g. that the operator
$B$ in Definition~\ref{def:admissibility} must be bounded. We refer to~\cite[Section 3]{Kar21} for a construction of an infinite-dimensional affine process involving unbounded $B$.\par
In order to identify {\it pure-jump} affine processes, we introduce an \emph{admissible parameter set} in the following definition. We think of $b$ as the \textbf{constant drift vector}, $B$ the \textbf{linear term in the drift},
$m$ the \textbf{constant jump measure}, and $\mu$ the \textbf{state-dependent jump measure}. \par 
Recall that Appendix~\ref{sec:integr-with-resp} summarizes theory on integration with respect to a Hilbert space valued measure.

\begin{definition}\label{def:admissibility}
An \emph{admissible parameter set}
$(b, B, m, \mu)$ 
consists of
\begin{enumerate}
\item \label{eq:m-2moment} a measure $m\colon\cB(\cHpluso)\to [0,\infty]$ such that
    \begin{enumerate}
    \item[(a)] $\int_{\cHpluso} \| \xi \|^2 \,\dm < \infty$ and
    \item[(b)] $\int_{\cHpluso}|\langle\chi(\xi),h\rangle|\,\dm<\infty$ for all $h\in\cH$  
   and there exists an element $I_{m}\in \cH$ such that $\langle
   I_{m},h\rangle=\int_{\cHpluso}\langle \chi(\xi),h\rangle\, m(\D\xi)$ for every $h\in\cH$\,;
 \end{enumerate}   
\item\label{eq:drift} a vector $b\in\cH$ such that
  \begin{align}\label{ass:b_positive}
    \langle b, v\rangle - \int_{\cHpluso} \langle \chi(\xi), v\rangle \,m(\D\xi) \geq 0\, \quad\text{for all}\;v\in\cHplus\,;
  \end{align}
\item \label{eq:affine-kernel} a $\cH^{+}$-valued measure 
$\mu \colon \mathcal{B}(\cHpluso) \rightarrow \cH^+$ such that
\begin{align*}
\int_{\cH^+\setminus \{0\}} \langle \chi(\xi), u\rangle \frac{\langle \dmu, x \rangle}{\| \xi \|^2 }< \infty,  
\end{align*}
for all $u,x\in \cH^{+}$ satisfying $\langle u,x \rangle = 0$\,;
 \item \label{eq:linear-operator} an operator $B\in \mathcal{L}(\mathcal{H})$ 
 with adjoint $B^{*}$ satisfying
\begin{align*}
    \left\langle B^{*}(u) , x \right\rangle 
    - 
    \int_{\cHpluso}
        \langle \chi(\xi),u\rangle 
        \frac{\langle \dmu, x \rangle}{\| \xi\|^2 }
    \geq 0,
  \end{align*}
  for all $x,u \in \cHplus$ satisfying $\langle u,x\rangle=0$.
\end{enumerate}
\end{definition}
\begin{remark}[Comparison to the finite-dimensional case]\label{rem:comparison_finite_dim_admissability}
 Definition~\ref{def:admissibility} above is analogous to the 
definition of an admissible parameter set for $\R^d_{+}$-valued processes 
see \cite[Def.\ 2.6]{DFS03}) and the case of positive semi-definite and symmetric matrices,
see \cite[Def.\ 2.3]{CFMT11}. However, as mentioned in the
introduction, we do not consider any diffusion terms in this work. A more subtle difference is that we require second moment conditions on the measures $\dm$ and 
$\frac{\dmu}{\| \xi \|^2}$, whereas no moment conditions are needed in the
finite-dimensional setting. These second moment conditions are a consequence of our \emph{generalized Feller approach}, for which we take the weight function $\rho = \| \cdot \|^2 + 1$. See Remark~\ref{rem:second-moment} for a detailed discussion regarding the necessity of these moment conditions to our approach.
\end{remark}

In what follows we will frequently use the following observation:
\begin{equation}\label{eq:exp_est}
\begin{aligned}
\forall \xi,u \in \cH^{+}\colon  \\
-\min(\langle \xi ,u \rangle, 1)\one_{\{ \| \xi \| > 1 \}}
& \leq 
 \e^{ - \langle \xi, u \rangle}
 -1
 + \langle \chi(\xi), u \rangle 
 \\ &
 \leq \tfrac{1}{2} |\langle \xi, u \rangle|^2 \one_{\{ \| \xi\| \leq 1 \}}
 \leq \tfrac{1}{2} \| \xi \|^2 \| u \|^2 \one_{\{ \| \xi \| \leq 1\}}.
 \end{aligned}
\end{equation}

Given admissible parameters $(b, B, m, \mu)$, we define $F\colon\cHplus\to \MR$ and $R\colon\cHplus\to\cH$, respectively, by
\begin{subequations}\label{eq:FR}
\begin{align}
F(u)&= \langle b,u\rangle-\int_{\cHpluso}\big(\E^{-\langle
  \xi,u\rangle}-1+\langle \chi(\xi),u\rangle\big)\dm, \\
  R(u)&= B^{*}(u)-\int_{\cHpluso}\big(\E^{-\langle
  \xi,u\rangle}-1+\langle \chi(\xi),
  u\rangle\big)\frac{\dmu}{\norm{\xi}^{2}}\,.
\end{align}
\end{subequations}
Note that the admissibility conditions (see Definition~\ref{def:admissibility}),
Corollary~\ref{cor:simpleintcond}, and~\eqref{eq:exp_est} ensure that $F$ and
$R$ are well-defined. We also have that $F$ and $R$ are continuous and grow at
most quadratically:
\begin{lemma}\label{lemma:RFcontinuous}
Let $(b,B,m,\mu)$ be an admissible parameter set conform Definition~\ref{def:admissibility} 
and let $F$ and $R$ be given by~\eqref{eq:FR}. 
Then $F$ and $R$ are continuous on $\cHplus$.
\end{lemma}
\begin{proof}
This follows immediately from~\eqref{eq:exp_est} and Theorem~\ref{thm:DCT}.
\end{proof}

\begin{lemma}\label{lemma:RFquadratic}
Let $(b,B,m,\mu)$ be an admissible parameter set conform Definition~\ref{def:admissibility} 
and let $F$ and $R$ be given by~\eqref{eq:FR}. 
Then for all $u\in \cH^+$ we have   
\begin{equation}\label{eq:Fquadratic}
 | F(u)|
 \leq 
 \left( \| b \|
    +
    \int_{\cH^+\setminus \{0\}}\|\xi\|^2\dm
 \right)
 (1+\|u\|^2)\,,
\end{equation}
and
\begin{equation}\label{eq:Rquadratic}
 \| R(u)\|
 \leq 
 \left( 
    \| B^{*} \|_{\cL(\cH)} 
    +
    \| \mu(\cH^+\setminus \{0\}) \|
 \right)
 (1+\|u\|^2)\,. 
\end{equation}
\end{lemma}

\begin{proof}
This follows immediately from the admissibility conditions,~\eqref{eq:exp_est},~\eqref{eq:normintest}, and~\eqref{eq:measuremonotone}.
\end{proof}

Inspired by the finite-dimensional theory, we consider a system of ordinary differential equations
associated with the admissible parameter set $(b,B,m,\mu)$ as introduced in the equations
\eqref{eq:Riccati-intro-phi}-\eqref{eq:Riccati-intro-psi}. The equations are commonly known as
the associated \textit{generalized Riccati equations} which is due to the typically quadratic
growth of $F$ and $R$ and by using the formulas for $F$ and $R$ in \eqref{eq:FR}, we write:
\begin{align}\label{eq:riccati-conc}
  \begin{cases}
    \frac{\partial \phi}{\partial t}(t,u)=F(\psi(t,u))\,,\quad t\geq 0;\qquad  \phi(0,u)=0\,, \\
    \frac{\partial \psi}{\partial t}(t,u)=R(\psi(t,u))\,,\quad t\geq 0;\qquad \psi(0,u)=u\,.
  \end{cases}  
\end{align}

\begin{definition}
Let $u\in \cHplus$. We say that $(\phi(\cdot,u), \psi(\cdot,u))\colon [0,\infty)\rightarrow \R \times \cH$ is a \emph{solution to~\eqref{eq:riccati-conc}} if $(\phi(\cdot,u), \psi(\cdot,u))$ is continuously differentiable, takes values in $\R^{+}\times \cH^{+}$, 
and satisfies \eqref{eq:riccati-conc}.
\end{definition}
For a transition semigroup $(P_{t})_{t\geq 0}$ defined on bounded measurable
functions on $\cHplus$ we recall the notion of a
\emph{weak generator} $(\cA,\dom(\cA))$ of $(P_{t})_{t\geq 0}$ (see
\cite[Definition 9.36]{PZ07}) i.e. $f\in C_{b}(\cHplus)$ belongs to
$\dom(\cA)$, whenever $\cA f(x)\df\lim\limits_{t\to 0+}\frac{P_{t}f(x)-f(x)}{t}$
exists for every $x\in\cHplus$; $\cA f\in C_{b}(\cHplus)$ and
\begin{align*}
P_{t}f(x)=f(x)+\int_{0}^{t}P_{s}\cA f(x)\D s,\quad x\in\cHplus.   
\end{align*}
The following theorem is our main result, it asserts the existence of affine
pure-jump processes on the cone of positive self-adjoint Hilbert-Schmidt
operators admitting for state-dependent jumps of infinite variation and
it specifies the form of the weak generator on a space of
functions containing the Fourier basis elements. For
the proof see Section \ref{proof:existence-affine-process}, which relies on
Section~\ref{sec:an-associated-class} and Section~\ref{sec:exist-affine-pure}.
\begin{theorem}\label{thm:existence-affine-process}
Let $(b, B, m, \mu)$ be an admissible parameter set (cf.\
Definition~\ref{def:admissibility}). Then there exist constants
$M,\omega \in [1,\infty)$ and a time-homogeneous $\mathcal{H}^+$-valued Markov process $X$  with transition semigroup
$(P_{t})_{t\geq 0}$ such that 
\begin{align}\label{eq:exp_bound}
\mathbb{E} [ \| X_t \|^2 | X_0 = x ] \leq M e^{\omega t} (\| x \|^2 +1)
\end{align}
and
\begin{align*}
  P_{t}\left(\E^{-\langle \cdot, u\rangle}\right)(x)=\E^{-\phi(t,u)-\langle x,\psi(t,u)\rangle},
\end{align*}
for all $t\geq 0$ and $u,x\in\cHplus$,  where $(\phi(\cdot,u),\psi(\cdot,u))$
is the unique solution to the associated generalized Riccati equations in
\eqref{eq:riccati-conc}. Moreover let $(\cA,\dom(\cA))$ be the weak
generator of $(P_{t})_{t\geq 0}$, then $\lin\set{\E^{-\langle \cdot,
    u\rangle}:\,u\in\cHplus}\subseteq \dom(\cA)$ and for every $f\in
\lin\set{\E^{-\langle \cdot, u\rangle}:\,u\in\cHplus}$ we have:
\begin{align}\label{eq:affine-generator-form}
\mathcal{\cA} f(x) &= \langle b +B(x) , f'(x) \rangle + \int_{\cHpluso}
                   \left(f(x+\xi) -f(x)-\langle \chi(\xi), f'(x)\rangle\right)\,\nu(x,\D\xi),
\end{align}
where $\nu(x,\D\xi)\df\dm +\frac{\langle\dmu,x\rangle}{\norm{\xi}^2}$.

\end{theorem}
\begin{proofoutline}
The proof is based on the approximation procedure that we conduct in detail
in Section~\ref{ssec:gener-fell-semigr}, where we work in the realm of generalized
Feller semigroups, see the preliminaries given in
Section~\ref{ssec:prel-gener-fell}. Here we limit ourselves to give a brief
outline of the proof that shall give a rough guidance for the upcoming sections and
condensing the main ideas therein. The detailed proof is then given in Section~\ref{proof:existence-affine-process}. Inspired by \cite{CT20}, we approximate the
Kolmogorov type operator $\cA$ in \eqref{eq:affine-generator-form}
by operators $(\cA^{(k)})_{k\in\MN}$ corresponding to processes of pure-jump
type with finite activity, i.e. for every $k\in\MN$ we replace the constant
jump measure $m(\D\xi)$ in formula \eqref{eq:affine-generator-form} by
$\one_{\set{\xi\geq 1/k}}m(\D\xi)$ and the linear jump measures $\mu(\D\xi)$ by
$\one_{\set{\xi\geq 1/k}}\mu(\D\xi)$. The approximation operators $\cA^{(k)}$
generate strongly continuous semigroups $(P_{t}^{(k)})_{t\geq 0}$ on a space
of functions, being weakly continuous with sub-quadratic growth, see
Proposition~\ref{prop:Gk_generator}. Having established the existence of
affine processes of pure-jump type associated with the strongly continuous
semigroups $(P_{t}^{(k)})_{t\geq 0}$, we next apply a Trotter-Kato type result
from \cite{CT20} to obtain the limiting semigroup $(P_{t})_{t\geq 0}$, see
Proposition~\ref{prop:affinejump-genFeller}. To this end we first need to
establish growth bounds on $(P_{t}^{(k)})_{t\geq 0}$, that are uniform in $k$,
see Proposition~\ref{prop:uniform-bound-k}.
This requires understanding the associated generalized Riccati equations
\eqref{eq:Riccati-intro-phi}-\eqref{eq:Riccati-intro-psi}. We provide
global existence and uniqueness results in
Section~\ref{sec:an-associated-class}. The crucial importance of the
associated ODEs is that they substitute for the Kolmogorov equations, hence
semigroup theoretic arguments involving the Kolmogorov type operators or the
abstract Cauchy problem can be reduced to ODE theoretic arguments.\\
Lastly, we apply a version of Kolmogorov's extension theorem (see
Theorem~\ref{thm:Kol-extension-theorem}) to the limiting semigroup $(P_{t})_{t\geq 0}$,
which then yields the existence of an underlying Markovian process. This
process associated via the semigroup to the operator $(\cA,\dom(\cA))$ is
the desired affine process identified by the admissible parameter set $(b, B, m, \mu)$. 
\end{proofoutline}

The second equation for $\psi(\cdot,u)$ in the generalized Riccati
equations~\eqref{eq:riccati-conc} is a non-linear differential equation on the cone of
positive self-adjoint Hilbert-Schmidt operators. 
This type of
infinite-dimensional differential equations has been of interest in the literature as they also show up e.g.~in optimal control problems and stochastic filtering theory
\cite{CZ95, GJP88, Lio71}. Hence several articles deal with the problem of numerical tractability of this type of equations. See, e.g.~\cite{Ros91} where Galerkin approximation and convergence theory was 
developed for operator-valued Riccati differential equations formulated in the
space of Hilbert-Schmidt operators and \cite{EEM19} where the author studied
a backward Euler approximation scheme and convergence results for this
type of equations. In a subsequent article \cite{Kar21}, we investigate the
Galerkin approximation further and draw a connection to matrix-valued affine processes.\\
An example of a stochastic volatility model where the covariance process is an affine Markov process on $\cH^+$ is the infinite-dimensional lift of the BNS model constructed in \cite{BRS15} to model forward rates in commodity markets. In \cite[Section 4]{cox2021infinitedimensional}
we constructed several other examples to model stochastic volatility in this context of forward rates in commodity markets
and we showed that our model class allow multiple modeling options for the instantaneous covariance process, including state-dependent jump intensity.

\section{Analysis of the generalized Riccati equations}\label{sec:an-associated-class}
In this section we investigate the generalized Riccati equations given by~\eqref{eq:riccati-conc}. In Subsection~\ref{ssec:sol_ODE} we introduce Lipschitz continuous approximations of the mappings $R$ and $F$
in~\eqref{eq:FR} and use these approximations to show existence and uniqueness
of a solution to~\eqref{eq:riccati-conc}. In Subsection~\ref{ssec:regul-with-resp} we establish regularity properties of $R$ and $F$
and use this to show that the solution map depends in a differentiable way on
its initial value. 
%
\subsection{Solving the generalized Riccati equations \texorpdfstring{\eqref{eq:riccati-conc}}{}}\label{ssec:sol_ODE}

The goal of this subsection is to prove the existence of a unique
solution to the generalized Riccati equations given an admissible parameter
set $(b,B,m,\mu)$.
A common approach in the finite-dimensional case, e.g. in the case of the cone
of positive semi-definite and symmetric matrices,
is to use a localisation argument exploiting the fact that the function $R$ is analytic on the interior of the cone.
Note, however, that in general $R$ fails to be Lipschitz continuous on the
boundary of the cone. The cone of positive self-adjoint Hilbert-Schmidt
operators has an empty interior, a property that is shared by many cones in
infinite dimensions. This has the consequence that localisation
arguments for solving equations~\eqref{eq:riccati-conc} on the
interior of $\MRplus\times\cHplus$ are not valid anymore.
Instead, for every $k\in \N$ we  introduce approximations $F^{(k)}$ of $F$ in equation \eqref{eq:Fk} and
$R^{(k)}$ of $R$ in equation \eqref{eq:Rk}, which involve only finite-activity
jump-measures, see~\eqref{eq:mk_muk} below. 
These approximations are Lipschitz continuous on $\cHplus$, and in
Proposition~\ref{prop:exist-uniq-solut} we show that the solution to the
generalized Riccati equations associated with $(b,B,m^{(k)},\mu^{(k)})$ converges to 
the (unique) solution to equation~\eqref{eq:riccati-conc}.\par 

We begin by introducing the approximating functions for $F$ and $R$: for $k\in\MN$ we set
\begin{align}\label{eq:mk_muk}
\dmk \coloneqq \one_{\{\norm{\xi}>1/k\}} \dm \quad  \text{and}\quad 
\dmuk \coloneqq \one_{\{\norm{\xi}>1/k\}} \dmu,
\end{align}
and we introduce the functions 
$F^{(k)}\colon \cH^{+}\rightarrow \R$ and 
$R^{(k)}\colon \cH^{+}\rightarrow \cH$ defined respectively as follows
\begin{align}
 F^{(k)}(u)&= \langle b,u\rangle-\int_{\cHpluso}\big(\E^{-\langle
 \xi,u\rangle}-1+\langle \chi(\xi),u\rangle\big)\dmk\,,\label{eq:Fk}\\ 
  R^{(k)}(u)&= B^{*}(u)-\int_{\cHpluso}\big(\E^{-\langle
  \xi,u\rangle}-1+\langle \chi(\xi),
  u\rangle\big)\frac{\dmuk}{\norm{\xi}^{2}}\,.\label{eq:Rk}
\end{align}
We denote the generalized Riccati equations associated to
$(b,B,m^{(k)},\mu^{(k)})$ by:
\begin{align}\label{eq:riccati-conc-k}
 \begin{cases}
 \frac{\partial \phi^{(k)}}{\partial t}(t,u)=F^{(k)}(\psi^{(k)}(t,u))\,,\quad t\geq 0;\qquad \phi^{(k)}(0,u)=0\,,\\
    \frac{ \partial \psi^{(k)}}{\partial t}(t,u)=R^{(k)}(\psi^{(k)}(t,u))\,,\quad t\geq 0;\qquad \psi^{(k)}(0,u)=u\,.
\end{cases}  
\end{align}
The notion of quasi-monotonicity will be needed 
to guarantee that the solution to~\eqref{eq:riccati-conc-k} stays in $\MRplus\times\cH^{+}$.
\begin{definition}\label{def:quasi-mono}
Let $(V,\|\cdot \|_V)$ be a Hilbert space 
and let $K\subset V$ be a self-dual cone. 
In addition, let $D\subseteq V$ and let $f\colon D\to V$, then $f$ is called
\emph{quasi-monotone with respect to $K$} if for all $v_1, v_2\in D$ satisfying $v_1 \leq_K v_2$
and for all $u\in K$ satisfying $\langle v_2 - v_1, u \rangle =0$ we have 
\begin{align*}
 \langle f(v_{2})-f(v_{1}),u\rangle\geq 0. 
\end{align*}
\end{definition}
Intuitively, quasi-monotone functions are pointing 'inwards' at the boundary
points, which ensures that solutions stay in a cone (see Theorem~\ref{thm:comparison-theorem}). For details 
on quasi-monotone functions on
Banach spaces and their connection to differential equations see~\cite[Section
5.3]{Dei77}.\par 

The following lemma states that the admissibility of 
parameters implies that $R^{(k)}$, $k\in \MN$, 
is quasi-monotone with respect to
$\cHplus$. The proof is analogous to the proof of~\cite[Lemma 5.1]{CFMT11}, we present an abridged version. 

\begin{lemma}\label{lemma:quasi-mono}
Let $B$ and $\mu$ satisfy the admissibility conditions~\ref{eq:affine-kernel} 
and~\ref{eq:linear-operator} in Definition \ref{def:admissibility}.
Then for all $k\in\MN$ the function $R^{(k)}$ given by~\eqref{eq:Rk} 
is quasi-monotone with respect to $\cHplus$. 
\end{lemma}

\begin{proof}
  The admissibility condition \ref{eq:linear-operator} in Definition~\ref{def:admissibility}
  (which makes sense thanks to condition~\ref{eq:affine-kernel} in Definition~\ref{def:admissibility})
  and the monotonicity of the exponential function imply the quasi-monotonicity of $R^{(k)}$.
\end{proof}
By removing the small jumps and since $m$ and $\mu$ have finite first moment, we obtain Lipschitz continuous mappings on $\cHplus$:
\begin{lemma}\label{lemma:mono-Lipschitz}
Let $B$ and $\mu$ satisfy the admissibility 
conditions~\ref{eq:affine-kernel} and~\ref{eq:linear-operator} in Definition \ref{def:admissibility}.
Let $k\in \MN$ and $R^{(k)}$ given by~\eqref{eq:Rk}. Then for all $u,v\in
\cHplus$ we have
\begin{align}\label{eq:Rk_Lipschitz}
\| R^{(k)}(u) - R^{(k)}(v) \|
\leq  
\left( \| B \|_{\cL(\cH)}
+ 
2k\| \mu(\cH^+ \setminus \{0\})\| \right) 
\| u-v\|
\end{align}
\end{lemma}

\begin{proof}
Observe that 
 for all $u,v,\xi \in \cH^{+}$ we have 
 \begin{equation*}
  \left|
    \e^{-\langle \xi, u \rangle } 
    - 
    \e^{- \langle \xi , v \rangle }
  \right|
  \leq 
  \| \xi \|
  \| u - v \|.
 \end{equation*}
Thus,~\eqref{eq:measuremonotone} and~\eqref{eq:normintest} imply that 
\begin{align*}
 \| R^{(k)}(u) - R^{(k)}(v) \|
 & \leq 
 \left\| 
    B^*(u-v)
 \right\|
 +
 \left\|\,
    \int_{\cH^+\setminus \{0\}\cap\{\frac{1}{k}<\|\xi\| < 1\}}
        \langle \xi, u - v \rangle 
    \frac{\dmu}{\| \xi \|^2 }
 \right\|
 \\ & \quad 
 + 
 \left\|\,
    \int_{\cH^+ \setminus \{0\}\cap\{\|\xi\| > \frac{1}{k}\}}
     (\E^{-\langle \xi,u\rangle} - \E^{-\langle \xi,v\rangle})  
    \frac{\dmu}{\| \xi \|^2 }
\right\|
\\ & \leq 
\left( \| B \|_{\cL(\cH)}
+ 
2k\| \mu(\cH^+\setminus \{0\}))\| \right) 
\| u-v\|.
\end{align*}
\end{proof}
Note that $R$ is typically not Lipschitz continuous on the whole $\cHplus$:
\begin{remark}\label{rem:R_loc_lip}
Note that 
\begin{equation}\begin{aligned}
 \left|
    e^{-\langle \xi, u \rangle} - e^{-\langle \xi, v \rangle} 
    +\langle \xi, u-v\rangle
 \right|
\leq 
 \left|
    \int_{\langle \xi, u \rangle}^{\langle \xi, v \rangle} s \,ds
 \right|
 \leq \|\xi\|^2 (\|u\|\vee \|v\|) \| u-v\|
 \end{aligned}
\end{equation}
for all $\xi,u,v\in \cH^+$. This implies that $R$ is in general Lipschitz
continuous only on bounded sets in $\cH^+$.
\end{remark}


By Lemmas~\ref{lemma:quasi-mono} and~\ref{lemma:mono-Lipschitz} we have that
$R^{(k)}$ is Lipschitz continuous on $\cH^{+}$ and quasi-monotone with respect
to $\cH^{+}$. Hence classical infinite dimensional ODE theory
guarantees the existence of a global solution to the equations~\eqref{eq:riccati-conc-k}: 
\begin{proposition}\label{prop:global_sol_finact}
Let $(b,B,m,\mu)$ be an admissible parameter set conform Definition~\ref{def:admissibility}
and let $R^{(k)}$, $k\in \MN$, be given by equation~\eqref{eq:Rk}.
Then for every $k\in\MN$ and $u\in\cHplus$ there exists a solution $(\phi^{(k)}(\cdot,u), \psi^{(k)}(\cdot,u))$ to \eqref{eq:riccati-conc-k}. Moreover,
\begin{equation}\label{eq:phik_monotone}
 \psi^{(k)}(t,u) \leq_{\cH^+} \psi^{(k)}(t,v),\quad \forall u,v\in \cH^+ \text{ satisfying }u\leq_{\cH^+}v,
\end{equation} 
for all $t\geq 0$ and
\begin{equation}\label{eq:phik_growthbound}
 \| \psi^{(k)}(t,u) - \psi^{(k)}(t,v) \| 
 \leq 
 \exp\left(\left( \| B \|_{\cL(\cH)}
    + 
    2k\| \mu(\cH^+ \setminus \{0\})\| 
 \right)t\right)
 \| u - v \|
\end{equation} 
for all $t\geq 0$ and $u,v\in \cH^+$.
\end{proposition}

\begin{proof}
 Let $k\in \N$.
By Lemma~\ref{lemma:mono-Lipschitz} the function $R^{(k)}$ is Lipschitz continuous on
$\cHplus$, by~\eqref{eq:Rk_Lipschitz} with $v=0$ the function $R^{(k)}$
satisfies the linear growth condition $\norm{R^{(k)}(u)}\leq \left( \| B
  \|_{\cL(\cH)}+ 2k\| \mu(\cH^+ \setminus \{0\})\| \right)\norm{u}$ and by
Lemma~\ref{lemma:quasi-mono} $R^{(k)}$ is quasi-monotone with respect to $\cHplus$,
thus by \cite[VI.3. Theorem 3.1 and Proposition 3.2]{Mar76} there exists a
unique global solution $\psi^{(k)}(\cdot,u)\colon [0,\infty)\to\cHplus$ to the
second equation of~\eqref{eq:riccati-conc-k}.
Now, setting $\phi^{(k)}(t,u) = \int_0^tF^{(k)}(\psi^{(k)}(s,u))\, \D s$, for
all $t\geq 0$, we obtain by continuity of $F^{(k)}$ and $\psi^{(k)}(\cdot, u)$
a solution $(\phi^{(k)}(\cdot,u), \psi^{(k)}(\cdot,u))$
to~\eqref{eq:riccati-conc-k} satisfying the inequality~\eqref{eq:phik_monotone}.
Finally, observe that Lemma~\ref{lemma:mono-Lipschitz} implies that
\begin{align*}
& \frac{\partial}{\partial t}\| \psi^{(k)}(t,u) - \psi^{(k)}(t,v) \|^2
\\ & = 2\left\langle \psi^{(k)}(t,u) - \psi^{(k)}(t,v), R^{(k)}(\psi^{(k)}(t,u)) - R^{(k)}(\psi^{(k)}(t,v))\right\rangle 
\\ & 
\leq 
2\left( \| B \|_{\cL(\cH)}
+ 
2k\| \mu(\cH^+ \setminus \{0\})\| \right)\| \psi^{(k)}(t,u) - \psi^{(k)}(t,v)\|^{2}.
\end{align*}
This and Gronwall's lemma implies the second inequality~\eqref{eq:phik_growthbound}.
\end{proof}
The next proposition guarantees the existence of a unique solution to
the original generalized Riccati equations~\eqref{eq:riccati-conc} on $[0,\infty)$.
First, we prove the following lemma:
\begin{lemma}\label{lem:Rk-R}
Let $B$ and $\mu$ satisfy the admissibility conditions~\ref{eq:affine-kernel} 
and~\ref{eq:linear-operator} in Definition~\ref{def:admissibility},
let 
$R^{(k)}$ and $R$ be respectively given by equation~\eqref{eq:Rk} and \eqref{eq:riccati-conc}.
Then for every $M>0$ we have
\begin{align*}
 \lim\limits_{k\to\infty}\sup_{u\in \cH^+: \| u \|\leq M}\norm{R^{(k)}(u)-R(u)}=0\,. 
\end{align*}
\end{lemma}
\begin{proof}
It follows immediately from~\eqref{eq:normintest} and~\eqref{eq:exp_est} that 
\begin{equation}
  \| R^{(k)}(u)-R(u) \|
  \leq 
  \| \mu(\{\xi \in \cH^{+}\colon \| \xi \| \leq \tfrac{1}{k} \}) \|
  \| u \|^2.
\end{equation}
The assertion follows from the above and the continuity of $ \mu $, see~\eqref{eq:meas_continuous}.
\end{proof}
\begin{proposition}\label{prop:exist-uniq-solut}
Let $(b,B,m,\mu)$ be an admissible parameter set conform Definition~\ref{def:admissibility}.
Then for every $u\in\cHplus$ there exists a
unique solution $(\phi(\cdot,u),\psi(\cdot,u))$ to~\eqref{eq:riccati-conc}.
Moreover,
\begin{align*}
\psi(t,u)\leq_{\cHplus} \psi^{(k)}(t,u)\quad  \forall\;k\in \N,\;t\geq 0\,\text{and}\; u\in \cH^+,   
\end{align*}
and $\psi(t,u)=\lim_{k\rightarrow \infty} \psi^{(k)}(t,u)$ for all $t\geq 0$ and $u\in \cH^+$, as well as 
\begin{equation}\label{eq:monotone_sol_riccati}
\psi(t,u)\leq_{\cHplus} \psi(t,v),\quad\forall\;t\geq 0\,\text{and}\;u,v\in\cH^+ \text{ with }
u\leq_{\cH^+} v,
\end{equation}
and
\begin{equation}\label{eq:expgrowth_sol_riccati}
 \| \psi(t,u) \| 
 \leq 
 \exp\left(\left( \| B \|_{\cL(\cH)}
    + 
    2\| \mu(\cH^+ \setminus \{0\})\| 
 \right)t\right)\| u \|,\quad \forall\;t\geq 0,\,u\in \cH^+.
\end{equation}
Finally, for all $M,T\geq 0$ there exists a $K(M,T)\geq 0$ 
such that for all $u,v\in \cH^+$ satisfying $\|u\|,\|v\|\leq M$ and all $t\in [0,T]$ it holds that 
\begin{equation}\label{eq:continuity_sol_riccati}
 \| \psi(t,u) - \psi(t,v)\| 
 \leq 
 K(M,T)\| u - v\|.
\end{equation}

\end{proposition}
\begin{proof}
First of all note that uniqueness of a solution follows 
from the fact that $R$ is Lipschitz continuous on bounded sets, see Remark~\ref{rem:R_loc_lip}.
  Observe that by~\eqref{eq:int_pos1},~\eqref{eq:exp_est}, and~\eqref{eq:Rk} we have, for all $u\in \cH^{+}$ and $k\in \N$,
  \begin{align}\label{eq:Rdecreasing}
    R^{(k)}(u)-R^{(k+1)}(u)
    &=\int_{\cHplus\cap \{\frac{1}{k+1}<\norm{\xi}\leq \frac{1}{k}\}}
    \left(
        \E^{-\langle u,\xi\rangle}-1+\langle \xi,u\rangle\right)
    \frac{\dmu}{\norm{\xi}^{2}}
    \nonumber\\ &
    \geq_{\cH^+} 0.
  \end{align}\par 
  Now fix $u\in \cH^{+}$. 
  By Proposition \ref{prop:global_sol_finact} we know that there exists a unique global
  solution $\psi^{(k)}(\cdot,u)$ to equation \eqref{eq:riccati-conc-k} for every $k\in\MN$. 
  This combined with~\eqref{eq:Rdecreasing} implies that for all $k\in \MN$ and $t\geq 0$ we have
  \begin{align*}
    \frac{\partial \psi^{(k+1)}}{\partial t}(t,u)-R^{(k+1)}(\psi^{(k+1)}(t,u))&=\frac{\partial \psi^{(k)}}{\partial t}(t,u)-R^{(k)}(\psi^{(k)}(t,u)) \\
    &\leq_{\cH^+} 
    \frac{\partial \psi^{(k)}}{\partial t}(t,u)-R^{(k+1)}(\psi^{(k)}(t,u)). 
  \end{align*}
  It follows from Lemma~\ref{lemma:mono-Lipschitz} and Theorem~\ref{thm:comparison-theorem} with $K=\cH^+$, $F=R^{(k+1)}$,
  $f=\psi^{(k+1)}(\cdot,u)$ and $g=\psi^{(k)}(\cdot,u)$ that 
  \begin{equation}\label{eq:phik_decreasing}
    \psi^{(k+1)}(t,u) \leq_{\cH^+} \psi^{(k)}(t,u)\,, \quad t\geq 0.
  \end{equation}
  As moreover $\psi^{(k)}(t,u)\geq_{\cH^{+}} 0$ for all $t\geq 0$ and $k\in \N$, the regularity of the 
  cone $\cH^+$ implies that for all $t\geq 0$ there exists a $\psi(t,u)\in \cH^+$ such that
  \begin{equation}\label{eq:phik_conv}
   \psi(t,u) = \lim_{k\rightarrow \infty} \psi^{(k)}(t,u).
  \end{equation}
  Note that by~\eqref{eq:phik_decreasing}, the 
  monotonicity of $\cH^+$, and the continuity of $\psi^{(1)}(\cdot,u)$ we
  have, for all $T>0$, 
  \begin{equation} \label{eq:upper-bound-phi}
    \sup_{k\in \N, s\in [0,T]} \| \psi^{(k)}(s,u) \| 
    \leq 
    \sup_{s\in [0,T]} \| \psi^{(1)}(s,u)  \| < \infty\,.
  \end{equation}
  It follows from this,~\eqref{eq:phik_conv}, the dominated convergence theorem, and Lemmas~\ref{lem:Rk-R} and~\ref{lemma:RFquadratic} that for all $t\geq 0$ we have 
  \begin{equation*}
  \begin{aligned}
   \psi(t,u) &= \lim_{k\rightarrow \infty} \psi^{(k)}(t,u)
   \\ 
   & = u + \lim_{k\rightarrow \infty} \int_{0}^{t} R^{(k)}(\psi^{(k)}(s,u))\D s
   \\   
   & = u + \lim_{k\rightarrow \infty} \int_{0}^{t} \left( R^{(k)}(\psi^{(k)}(s,u)) - R(\psi^{(k)}(s,u))\right)\D s
   \\ &
   \quad + \lim_{k\rightarrow \infty} \int_{0}^{t} R(\psi^{(k)}(s,u))\D s
   \\
   & = u + \int_{0}^{t} R(\psi(s,u))\D s.
   \end{aligned}
  \end{equation*}
 The equation above combined with Lemma~\ref{lemma:RFquadratic} implies that the map $\psi(\cdot,u)$ is continuous,
 whence Lemma~\ref{lemma:RFcontinuous} and the fundamental theorem of calculus 
 imply that $\psi(\cdot,u)\in C^1([0,\infty),\cH)$ and 
 \begin{equation}
  \frac{\partial\psi}{\partial t}(t,u) = R(\psi(t,u)),\quad t\geq 0;\qquad \psi(0,u)=u.
 \end{equation}
 Moreover, the continuity of $F$ and of $\psi(\cdot,u)$ ensures that by setting 
 \begin{equation}
  \phi(t,u) = \int_{0}^{t} F(\psi(s,u))\D s, \quad t\geq 0,
 \end{equation}
 we obtain that $(\phi(\cdot,u),\psi(\cdot,u))$ is a solution to~\eqref{eq:riccati-conc}.\par 
 Next, note that~\eqref{eq:monotone_sol_riccati} follows from~\eqref{eq:phik_monotone} and~\eqref{eq:phik_conv}.
 Moreover,~\eqref{eq:expgrowth_sol_riccati} follows from~\eqref{eq:phik_growthbound} with $k=1$,~\eqref{eq:phik_decreasing},~\eqref{eq:phik_conv}, and the fact 
 that $\psi^{(1)}(t,0)\equiv 0$. 
 Finally,~\eqref{eq:continuity_sol_riccati} follows from the Lipschitz continuity
 of $R$ on bounded sets (see Remark~\ref{rem:R_loc_lip}),~\eqref{eq:expgrowth_sol_riccati}, and 
 the same reasoning as we used to obtain~\eqref{eq:phik_growthbound}.
\end{proof}

\subsection{Regularity with respect to the initial value of the
  solution}\label{ssec:regul-with-resp}
Having established the existence of a unique solution to~\eqref{eq:riccati-conc}, we now turn to the regularity of the solution with respect to the initial value. To this end we first must introduce a fitting concept of differentiability:
\begin{definition}\label{def:one-sided-derivative}
  Let $X$ and $Y$ be Banach spaces and
  $D\subseteq X$ a convex subset. We say that
  a function $f\colon D\subseteq X\to
Y$ has a \emph{one-sided derivative at $x\in D$ in the direction $v\in
X$}, whenever $x+\lambda v\in D$ for all $\lambda$ sufficiently small and the limit
\begin{align*}
 \lim_{\lambda\to 0+}\frac{f(x+\lambda v) -f(x)}{\lambda}, 
\end{align*}
exists in $Y$. We denote this limit by $\D_{+}f(x)(v)$.
We define the second one-sided derivative in $x\in D$ in direction $(v,w)\in X\times X$ as
\begin{align*}
 \lim_{\lambda\to 0+}\frac{\D_{+}f(x+\lambda w)(v) -\D_{+}f(x)(v)}{\lambda}, 
\end{align*}
whenever $x+\lambda w\in D$ and $\D_{+}f(x+\lambda w)(v)$ exists
for all $\lambda$ sufficiently small and moreover the limit exists in $Y$. We
denote the second one-sided derivative of $f$ at $x$ in directions $(v,w)$ by
$\D^{2}_{+}f(x)(v,w)$. 
\end{definition}


\begin{lemma}\label{lem:R-one-sided-dif} 
  Let $(b,B,m,\mu)$ be an admissible parameter set conform Definition~\ref{def:admissibility} 
and let $F$ and $R$ be given by~\eqref{eq:FR}. 
For $u\in\cH^+$ define $\D R(u) \in \cL(\cH)$ by
\begin{equation}\label{eq:def_dR}
 \D R (u)v = B^{*}(v)+\int_{\cHpluso}\langle \xi, v\rangle\E^{-\langle \xi,u\rangle}-\langle \chi(\xi),
                 v\rangle \frac{\dmu}{\norm{\xi}^{2}},\quad v\in \cH,
\end{equation}
and $\D F(u)\in\cL(H,\MR)$ by
\begin{equation}\label{eq:def_dF}
 \D F(u)v= \langle b, v\rangle+\int_{\cHpluso}\langle \xi, v\rangle\E^{-\langle \xi,u\rangle}-\langle \chi(\xi),
                 v\rangle \dm,\quad v\in \cH.
\end{equation}
Moreover define $\D^2 R(u) \in \cL^{(2)}( \cH \times \cH, \cH)$ by
\begin{equation}\label{eq:def_d2R}
 \D^2 R (u)(v,w) = -\int_{\cHpluso}\langle \xi, v\rangle\langle \xi, w\rangle\E^{-\langle \xi,u\rangle} \frac{\dmu}{\norm{\xi}^{2}},\quad v,w\in\cH.
\end{equation}
and $\D^{2}F(u)\in \cL^{(2)}( \cH \times \cH, \MR)$ by
\begin{equation}\label{eq:def_d2F}
 \D^2 F (u)(v,w) = -\int_{\cHpluso}\langle \xi, v\rangle\langle \xi,
 w\rangle\E^{-\langle \xi,u\rangle} \dm,\quad v,w\in\cH.
\end{equation}
Then the operator $\D R(u)$ is quasi-monotone for all $u\in \cH^+$, and 
for all $u_0,u_1\in\cH^+$ and $v,w\in \cH$ we have 
\begin{align}
 \| \D R(u_0) (v) \| 
 & 
 \leq 
 \| B^* \|_{\cL(\cH)}\|v\|
 +
 \left\|
 \mu(\cH^+\setminus \{0\})
 \right\|
 (1+\| u_0\|)
 \| v\| \label{eq:dR_bdd}
\\
 \| \D R(u_0)(v) - \D R(u_1)(v) \|
 & \leq 
 \left\|
    \mu(\cH^+\setminus \set{0})
 \right\| 
 \| u_0 - u_1\|
 \| v\|,\label{eq:dR_lipschitz}
 \\ 
 \| \D^2 R(u_0)(v,w) \|
 & \leq 
 \left\|
 \mu(\cH^+\setminus \{0\})
 \right\| \| v \| \| w \|, \label{eq:d2R_bdd}
\end{align}
and $u\mapsto \D^2 R(u)(v,w)$ is continuous.
Moreover, $F$ and $R$ are two-times one-sided differentiable 
in $u$ in the direction $(v,w)$ for all $u,v,w\in \cH^{+}$, and for all $u,v,w\in \cH^{+}$ we have:
  \begin{align}\label{eq:R-one-sided-derivative}
   \D_{+}R(u)(v)&=\D R(u)v,
    \\    
  \label{eq:R-one-sided-dif-second}
    \D^{2}_{+}R(u)(v,w)& =\D^2 R(u)(v,w),
    \\    
  \D_{+}F(u)(v)&=\D F(u)v, \label{eq:F-one-sided-derivative}
    \\
   \D^{2}_{+}F(u)(v,w)&=\D F(u)(v,w). 
                       \label{eq:F-one-sided-dif-second}
  \end{align}   
\end{lemma}

\begin{proof}
The quasi-monotonicity of $\D R$ follows directly from the admissibility assumption. As 
\begin{equation*}
 \left|
    \langle \xi, v\rangle\E^{-\langle \xi,u\rangle}-\langle \chi(\xi),v\rangle
 \right|
 \leq 
 \| \xi \| \|v\| (\one_{\{\| \xi\| > 1 \}} + \| \xi \| \| u \| \one_{\{\|\xi\|\leq 1\}} )
\end{equation*}
for all $u,\xi\in \cH^+$ and all $v\in \cH$, we obtain~\eqref{eq:dR_bdd}. Estimate~\eqref{eq:dR_lipschitz}
is obtained similarly, estimate~\eqref{eq:d2R_bdd} is immediate from the definition, and the 
continuity of $u\mapsto \D^2 R(u)(v,w)$ follows from the dominated convergence theorem (Theorem~\ref{thm:DCT}).\par
We next confirm the asserted differentiability of the map $u\mapsto
  R(u)$. Let $u,v\in\cHplus$ then
\begin{align}
  \D_{+}R(u)(v)&= \lim_{\lambda \to  0+}\frac{R(u+\lambda v)-R(u)}{\lambda} \nonumber\\
               &=B^{*}(v)-\lim_{\lambda\to 0+}\int_{\cHpluso}\frac{\E^{-\langle
                 \xi,u+\lambda v\rangle}-\E^{-\langle
                 \xi,u\rangle}}{\lambda}+\langle \chi(\xi),
                 v\rangle\frac{\mu(\D\xi)}{\norm{\xi}^{2}}\nonumber \\
               &=B^{*}(v)-\int_{\cHpluso}\lim_{\lambda\to 0+}\frac{\E^{-\langle
                 \xi,u+\lambda v\rangle}-\E^{-\langle
                 \xi,u\rangle}}{\lambda}+\langle \chi(\xi),
                 v\rangle\frac{\mu(\D\xi)}{\norm{\xi}^{2}} \label{eq:R-one-sided-derivative-interchange}\\
               &=B^{*}(v)+\int_{\cHpluso}\langle \xi, v\rangle\E^{-\langle \xi,u\rangle}-\langle \chi(\xi),
                 v\rangle \frac{\mu(\D\xi)}{\norm{\xi}^{2}}.\nonumber
\end{align}
where the interchange of the integral and the limit in equation
\eqref{eq:R-one-sided-derivative-interchange} is justified, since
$\lambda\mapsto \E^{-\langle u+\lambda v,\xi \rangle}$ is a convex
mapping, hence its differential quotient is non-decreasing in $\lambda$ and non-negative 
and thus we can apply the monotone convergence theorem to obtain that the one-sided
derivative of $R$ exists in $u$ in the direction $v$
and~\eqref{eq:R-one-sided-derivative} holds. An analogous derivation for $F$
leads to equation~\eqref{eq:F-one-sided-derivative}.\par 
The proof that the \emph{second} one-sided directional derivative of both $F$
and $R$ exist and
that~\eqref{eq:R-one-sided-dif-second}--\eqref{eq:F-one-sided-dif-second} hold
is again analogous. Note in particular that for the existence of the second derivatives we use that the measures $\dm$ and $\frac{\dmu}{\|\xi\|^2}$ have finite second moments.
\end{proof}

Proposition~\ref{prop:gateaux-diff} below states that the solution
$(\phi(\cdot,u),\psi(\cdot,u))$ to \eqref{eq:riccati-conc} is such that the
mappings $u\mapsto \psi(t,u)$ and $u\mapsto \phi(t,u)$ are twice one-sided differentiable in $0$ in all directions. The techniques to prove this are well-known, however, as we are dealing with a non-standard concept of differentiability we provide the details of the proof in Appendix~\ref{sec:proof-lemma-refl}.

\begin{remark}
In fact, one can prove that $u\mapsto \psi(t,u)$ and $u\mapsto \phi(t,u)$
are twice one-sided differentiable in $u$ for every $u\in \cH^+$, in every direction $(v,w)\in \cH^+\times \cH^+$. We do not need this, but we \emph{do} need the existence of the first derivative in $u\in \cH^+$ for $u$ 
sufficiently small in order to obtain the second derivative. See also Appendix~\ref{sec:proof-lemma-refl}.
\end{remark}
\begin{proposition}\label{prop:gateaux-diff}
Let $(b,B,m,\mu)$ be an admissible parameter set conform Definition~\ref{def:admissibility}, 
for every $u\in\cHplus$ let $(\phi(\cdot,u),\psi(\cdot,u))$ be the solution to \eqref{eq:riccati-conc}, and let $\D R$, $\D F$, $\D^2 R$, and $\D^2 F$ be defined by~\eqref{eq:def_dR}--\eqref{eq:def_d2F}.
Then the maps $u\mapsto \psi(t,u)$ and $u\mapsto \phi(t,u)$
are twice one-sided differentiable in $0$ in all directions $(v,w)\in\cHplus \times \cHplus$. Moreover, 
$\D_{+}\psi(t,0)(v),\D_+^2\psi(t,0)(v,w)\in \cH^+$ for all $v,w\in \cH^+$ and the mappings $t\mapsto
\D_{+}\phi(t,0)(v)$ and $t\mapsto\D_{+}\psi(t,0)(v)$  solves the following pair of differential
equations:
  \begin{align}
    \frac{\partial}{\partial t} \D_{+}\phi(t,0)(v)= \D F(0)\big(\D_{+}\psi(t,0)(v)\big),
    \quad t\geq 0;\quad \D_{+}\phi(0,0)(v)=0, \label{eq:dif-eq-phi-1}\\ 
    \frac{\partial}{\partial t} \D_{+}\psi(t,0)(v)= \D R(0)\big(\D_{+}\psi(t,0)(v)\big),
    \quad t\geq 0;\quad \D_{+}\psi(0,0)(v)=v, \label{eq:dif-eq-psi-1}
  \end{align}
  Moreover, the
  mappings $t\mapsto\D^{2}_{+}\psi(t,0)(v,w)$ and
  $t\mapsto\D^{2}_{+}\phi(t,0)(v,w)$ solve the following pair of differential equations:
  \begin{align}
   \frac{\partial}{\partial t} \D^{2}_{+}\phi(t,0)(v,w)
   &= \D^{2}F(0)(\D_{+}\psi(t,0)(v),\D_{+}\psi(t,0)(w))
   \nonumber\\ & \quad + \D F(0)\big(\D^{2}_+\psi(t,0)(v,w)\big),
                 \quad t\geq 0;\quad \D_{+}^2\phi(0,0)(v,w)=0,\label{eq:dif-eq-phi-2}\\
    \frac{\partial}{\partial t} \D^{2}_{+}\psi(t,0)(v,w)&=
    \D^{2} R(0)\big(\D_{+}\psi(t,0)(v),\D_{+}\psi(t,0)(w)\big) 
   \nonumber\\ & \quad + 
    \D R(0)(\D^{2}_{+}\psi(t,0)(v,w)),
    \quad t\geq 0;\quad \D_{+}^2\psi(0,0)(v,w)=0.\label{eq:dif-eq-psi-2}
  \end{align}
\end{proposition}
\begin{proof}
See Appendix~\ref{sec:proof-lemma-refl}.
\end{proof}
For $u=0$ we derive explicit formulas for the solutions to the pairs of differential
equations in \eqref{eq:dif-eq-psi-1}
and
\eqref{eq:dif-eq-psi-2}
of Proposition
\ref{prop:gateaux-diff}, as those will be needed for proving Lemma
\ref{lem:uniform-bound-k} in the approximating case and for Proposition \ref{prop:explicit-formula} below. First, note that
\begin{align*}
 \D_{+}R(0)(v)=B^{*}(v)+\int_{\cHplus\cap \{\norm{\xi}\geq 1\}}\langle \xi, v\rangle\frac{\dmu}{\norm{\xi}^{2}}. 
\end{align*}
Recall the definition of $\D R(0)$ from~\eqref{eq:def_dR}.
The solution of equation~\eqref{eq:dif-eq-psi-1} is then given by
\begin{align}\label{eq:variational-solution-3-0}
\D_{+}\psi(t,0)(v)&= \E^{t \D R(0)}v.
\end{align}
By inserting formula~\eqref{eq:variational-solution-3-0} into
equation~\eqref{eq:dif-eq-psi-2} (note that $e^{t \D R(0)}v \in \cH^+$) and solving this inhomogeneous linear
equation we obtain
\begin{align}\label{eq:variational-solution-3-1}
 \D^{2}_{+}\psi(t,0)(v,w)
 &=
 \int_{0}^{t}
    \E^{(t-s)\D R(0)}
    \D^2 R(0)
    (\E^{s \D R(0)}v, \E^{s \D R(0)}w)
 \D s.
\end{align}
%
%
\section{Existence of affine pure-jump processes in
  \texorpdfstring{$\cHplus$}{the space of positive self-adjoint Hilbert-Schmidt operators}}\label{sec:exist-affine-pure}
In this section we use the well-posedness and regularity results of the
generalized Riccati equations~\eqref{eq:riccati-conc} from Section
\ref{sec:an-associated-class} to show the existence of an affine process in
$\cHplus$ associated to a given admissible parameter set $(b, B, m, \mu)$
conform Definition \ref{def:admissibility}. Due to the lack of local compactness of the underlying state space, standard Feller theory cannot be employed in our context and 
we use the theory of generalized Feller processes as introduced in \cite{DT10}.
The existence proof is based on the
approximation procedure roughly sketched at the end of Section
\ref{sec:setting-main-result}. In this section we rigorously build up this
approximation procedure in the \emph{generalized Feller} setting. 
 Essentially,
we approximate the transition semigroup $(P_{t})_{t\geq 0}$, that can be
associated to an affine process in $\cHplus$ with infinite-activity jump
behavior, by simpler transition semigroups corresponding
to affine finite-activity jump processes. The considered semigroups are
strongly continuous semigroups on a certain Banach space of real functions being
weakly-continuous on compact sets and having at most quadratic growth in the
tails. We briefly introduce the \textit{generalized Feller} setting, that is
we  define \emph{generalized Feller} semigroups and processes in Section
\ref{ssec:prel-gener-fell} and consequently in Section
\ref{ssec:gener-fell-semigr} we apply approximation results from the theory of
strongly continuous semigroups adapted to the generalized Feller setting by \cite{CT20}.   
\subsection{Preliminaries: generalized Feller semigroups}\label{ssec:prel-gener-fell}  
We recall the concept of generalized Feller semigroups introduced in
\cite{DT10} and further developed in \cite{CT20}.\\ Throughout this section let $(Y,\tau)$ be a complete regular Hausdorff space. 

\begin{definition}
A function $\rho\colon Y \to (0,\infty)$ such that for every $R>0$ the set
$K_{R}\df\set{x\in Y:\;\rho(x)\leq R}$ is compact is called an \emph{admissible
weight function}. The pair $(Y,\rho)$ is called \emph{weighted space}. 
\end{definition}

Let $\rho\colon Y \rightarrow (0,\infty)$ be an admissible weight function. 
For $f\colon Y \rightarrow \MR$ we define $\| f \|_{\rho} \in [0,\infty]$ by
\begin{equation}\label{eq:weighted_norm}
 \| f \|_{\rho} \df \sup_{x\in Y} \tfrac{|f(x)|}{\rho(x)}.
\end{equation}
Note that $\| \cdot \|_{\rho}$ defines a norm on the vector space 
$
  B_{\rho}(Y)\df
  \set{
    f\colon Y\to \MR 
    \colon 
    \| f \|_{\rho} < \infty
    }
$
which renders $(B_{\rho}(Y),\|\cdot \|_{\rho})$ a Banach space.
Recall that $C_b(Y)$ denotes the space of bounded $\R$-valued $\tau$-continuous functions on $Y$. 
As any admissible weight function satisfies $\inf_{x\in Y}\rho(x)>0$, we have that $C_b(Y) \subseteq  B_{\rho}(Y)$.
\begin{definition}
 We define $\cB_{\rho}(Y)$ to be the closure of $C_{b}(Y)$ in $B_{\rho}(Y)$.
 \end{definition}
The following useful characterization of $\cB_{\rho}(Y)$ is proven
in~\cite[Theorem 2.7]{DT10}:
\begin{theorem}\label{thm:charac_Brho}
 Let $(Y,\rho)$ be a weighted space. Then $f\in \cB_{\rho}(Y)$
 if and only if $f\rvert_{K_{R}}\in C(K_{R})$ for all $R>0$
and 
\begin{equation}\label{eq:function-growth}
\lim_{R\to\infty}\sup_{x\in Y\setminus K_{R}}\tfrac{|f(x)|}{\rho(x)}=0\,.
\end{equation}
\end{theorem}
%

We can now present the definition of a generalized Feller semigroup, as introduced in \cite[Section 3]{DT10}.
\begin{definition}\label{def:genFeller-semigroup}
A family of bounded linear operators $(P_t)_{t\geq 0}$ in $\cL(\cB_{\rho}(Y))$
is called a {\it generalized Feller semigroup (on $\cB_{\rho}(Y)$)}, if
\begin{enumerate}
\item \label{semigroup1} $P_0 = I$, the identity on $\cB_{\rho}(Y)$, 
\item  \label{semigroup2}  $P_{t+s} = P_t P_s$ for all $t,s \geq 0$, 
\item  \label{semigroup3}  $\lim\limits_{t\to 0+}P_{t}f(x)=f(x)$ for all $f\in \cB_{\rho}(Y)$ and $x\in Y$,
\item   \label{semigroup4} there exist constants $C\in\MR$ and $\varepsilon>0$ such that $\norm{P_{t}}_{\cL(\cB_{\rho}(Y))}\leq C$ for all $t\in [0,\varepsilon]$,
\item  \label{semigroup5}  $(P_{t})_{t\geq 0}$ is a positive semigroup, i.e., $P_{t}f\geq 0$ for all $t\geq 0$ and
  for all $f\in \cB_{\rho}(Y)$ satisfying $f\geq0$.
\end{enumerate}
\end{definition}
By \cite[Theorem 3.2]{DT10} any generalized Feller semigroup is strongly continuous.
Moreover, generalized Feller semigroups allow for a Kolmogorov extension theorem, see \cite[Theorem 2.11]{CT20}
for a proof:
\begin{theorem}\label{thm:Kol-extension-theorem}
Let $(P_t)_{t \geq 0}$ be a generalized Feller semigroup on $\cB_{\rho}(Y)$ satisfying $P_t1=1$ for all $t\geq 0$. 
Then there exists a filtered measurable space $(\Omega, (\mathcal{F}_t)_{t\geq0})$ with a
right-continuous filtration and a family of functions $X_t\colon \Omega \rightarrow Y$, $t\geq 0$, such that $X_t$ is $\cF_{t}$ measurable for all $t\geq 0$ and for any initial value $x \in Y$ there exists a probability measure $\mathbb{P}_{x}$ such that
\begin{equation}\label{eq:GF-process}
\mathbb{E}_{\mathbb{P}_{x}} [f(X_t)] = P_tf(x)
\end{equation}
for every $t\geq 0$ and every $f \in \cB_{\rho}(Y)$. Moreover, for all $x\in Y$ the process $(X_t)_{t\geq 0}$ 
is a time-homogeneous $\mathbb{P}_{x}$-Markov process, i.e., for all $x\in Y$, $0\leq s < t$, $f\in \cB_{\rho}(Y)$ we have 
\begin{align}\label{eq:Kol-extension-theorem-1}
\mathbb{E}_{\mathbb{P}_{x}}[f(X_t)\,|\, \cF_s] = P_{t-s}f(X_s),  
\end{align}
almost surely with respect to $\mathbb{P}_{x}$.
\end{theorem}
Let $(P_t)_{t\geq 0}$ be a generalized Feller semigroup satisfying $P_t 1 = 1$ for all $t\geq 0$. 
The process $(X_t)_{t\geq 0}$, the existence of which is guaranteed by Theorem~\ref{thm:Kol-extension-theorem},
is called a {\it generalized Feller process} with initial value $x$ with respect to the measure $\mathbb{P}_{x}$.\par  
From now on we write $\mathbb{E}_x$ for expectations with respect to the
probability measure $\mathbb{P}_{x}$.
\begin{remark}\label{rem:semigroup-on-b-rho}
Let $(P_{t})_{t\geq 0}$ be a generalized Feller semigroup and let $x\in Y$, then by a Riesz representation-type result (see~\cite[Theorem 2.4 and Remark 2.8]{CT20}) $P_{t}\rho(x)\in \R$ can be defined by the integral of $\rho$ with respect to the measure representing the linear functional $f\mapsto P_{t}f(x)$, $f\in\cB_{\rho}(Y)$. Moreover, as there exist $M>1$, $\omega\in \R$ such that $|P_{t}f(x)|\leq M\exp(\omega t)\rho(x)\norm{f}_{\rho}$ for all $f\in \cB_{\rho}(Y)$, we obtain
\begin{align}\label{eq:growth-bound-expectation}
P_{t}\rho\leq M\exp(\omega t)\rho  
\end{align}
for $t\geq 0$.
If moreover $(P_{t})_{t\geq 0}$ is associated to a Markov process
$(X_{t})_{t\geq 0}$ such that equation~\eqref{eq:GF-process} holds, we obtain:
\begin{align*}
\mathbb{E}_{x} [\rho(X_t)] = P_t\rho(x)\leq M\exp(\omega t)\rho.  
\end{align*}
This can be seen by equation \eqref{eq:growth-bound-expectation} and a
monotone convergence argument by choosing for every $n\in\MN$ the approximations
$\rho_{n}=\sum_{i=1}^{n}\langle \cdot, e_{i}\rangle^{2} \wedge n\in\cB_{\rho}(Y)$,
where $(e_{i})_{i\in\MN}$ is an ONB of $\cH$, then $\rho_{n}\to \rho$ in
pointwise as $n\to\infty$ and $\rho_{n}\leq \rho_{n+1}$ for all $n\in\MN$. 
\end{remark}
\subsection{Approximation of semigroups associated to affine processes in
  \texorpdfstring{$\cHplus$}{the cone of positive self-adjoint Hilbert-Schmidt
  operators}}\label{ssec:gener-fell-semigr}
We equip the Hilbert space $\cH$ with its weak topology $\sigma(\cH, \cH')$ 
(which, by the Riesz representation theorem, is the weak-$*$-topology). Note that as $\cHplus$ is self-dual, it is closed in $(\cH, \sigma(\cH,\cH'))$. 
For brevity of notation we let $\cH^+_{\textnormal{w}}$ denote the complete
regular Hausdorff space $(\cH^+, \sigma(\cH,\cH')_{\cH^+})$, where
$\sigma(\cH,\cH')_{\cH^+}$ denotes the relative topology $\sigma(\cH,\cH')$ on $\cHplus$. 
In addition, we define $\rho \colon \cH^+ \rightarrow \R$ by 
\begin{equation}\label{eq:def_rho}
 \rho(x) \df 1 + \| x \|^2\,,\quad x\in \cH^+\,,
\end{equation} 
and observe that $\rho$ is an admissible weight function on
$\cH^+_{\textnormal{w}}$ by the Banach-Alaoglu theorem, i.e.,
$(\cHplus_{\textnormal{w}},\rho)$ is a weighted space. Note that for
every $R>0$, the pre-image $\set{x\in \cHplus:\;\rho(x)\leq R}$ is compact in
$\cHplus$ equipped with the norm topology, if and only if $\cH$ is finite-dimensional.
As we assume throughout the article that $\cH$ is infinite-dimensional, we see
that $\rho$ is not an admissible weight function in the norm topology.\par{}
The linear span of the set of Fourier basis
elements $\set{\E^{-\langle\cdot, u\rangle}:\;u\in\cHplus}$ is denoted by
\begin{equation}\label{eq:span-Fourier}
\cD\coloneqq \operatorname{lin}\left(\left\{\E^{-\langle \cdot, u\rangle}:\;u\in\cHplus\right\}\right)\,.
\end{equation}
The relevance of this set lies in the following lemma.
\begin{lemma}\label{lem:dense_subset_Brho}
The set $\cD$ is dense in  $\cB_{\rho}(\cHplus_{\textnormal{w}})$.
\end{lemma}
\begin{proof}
It suffices to prove that for every $\eps>0$ and every $f\in C_b(\cH^+_{\textnormal{w}})$ 
there exists an $f_{\eps}\in \cD$ such that $\| f - f_{\eps}\|_{\rho} < \eps$. 
To this end, observe that for every $\eps>0$ and every $f\in C_b(\cH^+_{\textnormal{w}})$ there 
exists an $R>0$ such that $\sup_{x\in \cH^+, \| x \| > R} \tfrac{f(x)}{\rho(x)} < \tfrac{\eps}{2}$,
and apply Stone-Weierstrass to $C(\cH^+_{\textnormal{w}} \cap \{x\in \cH^+\colon \| x \| \leq R\})$.
\end{proof}

\begin{corollary}\label{cor:sepa}
The space $\cB_{\rho}(\cH^+_{\textnormal{w}})$ is separable.
\end{corollary}

\begin{proof}
Let $U$ be a countable dense set in $(\cH^+,\|\cdot \|)$ (recall from Section~\ref{ssec:setting} that 
$\cH^+$ is separable). 
Then by Lemma~\ref{lem:dense_subset_Brho} the set $\big\{\sum\limits_{j=1}^{n} q_j \E^{-\langle \cdot, u_{j}\rangle}:\; n\in \N,\, q_j \in \MQ,\, u_j \in U\big\}$
is dense in $\cB_{\rho}(\cH^+_{\textnormal{w}})$. 
\end{proof}

Throughout the remainder of this section let $(b,B,m,\mu)$ be an admissible parameter 
set, see Definition~\ref{def:admissibility}. First, we define for $k\in\N$, $\tilde{B}^{(k)} \in \cL(\cH)$ and $\tilde{b}^{(k)} \in \cH^+$ by
\begin{align*}
\tilde{B}^{(k)} (x) &\df B(x) -  \int_{\cHplus\cap\{0<\norm{\xi}\leq 1\}}\xi\, \frac{ \langle \dmuk,x\rangle}{\norm{\xi}^{2}}\,,\quad x\in \cH^+\,,\\
\tilde{b}^{(k)}&\df b-\int_{\cHplus\cap\{0<\norm{\xi}\leq 1\}}\xi \,\dmk\,,
\end{align*}
where $m^{(k)}$ and $\mu^{(k)}$
are as defined in~\eqref{eq:mk_muk}. Note that the fact that $B\in \cL(\cH)$ and that $\mu$ 
is an $\cH^+$-valued measure, as well as~\eqref{eq:normintest}~and
\eqref{eq:mk_muk} ensure that $\tilde{B}^{(k)}\in \cL(\cH)$ is
well-defined. Moreover,~\ref{eq:m-2moment} in
Definition~\ref{def:admissibility} and~\eqref{ass:b_positive} 
ensure that $\tilde{b}^{(k)} \in \cH^+$ is well-defined.
For $x\in \cH^+$ and $k\in \N$ we consider the following deterministic
equation in differential form:
\begin{align}\label{eq:model-deterministic-k}
  \begin{cases}
    \D \mathbf{x}^{(x,k)}_{t}&=\big(\tilde{b}^{(k)} +   \tilde{B}^{(k)}
    (\mathbf{x}^{(x,k)}_t)\big)\D t, \qquad t\geq 0,\\
    \mathbf{x}^{(x,k)}_{0}&= x.
 \end{cases}
\end{align}
Standard infinite-dimensional ODE theory ensures that for all $x\in \cH^+$ and $k\in \N$
the unique 
classical solution to~\eqref{eq:model-deterministic-k} is given by
\begin{equation}\label{eq:solution-deterministic}
\mathbf{x}^{(x,k)}_{t}\coloneqq \e^{t \tilde{B}^{(k)}}x+\int_{0}^{t}\e^{(t-s) \tilde{B}^{(k)}} \tilde{b}^{(k)}\D s\,,\quad t\geq 0\,.
\end{equation}
The following lemma provides some properties of $\mathbf{x}^{(x,k)}$, $x\in \cH^+$, $k\in \N$.
\begin{lemma}\label{lemma:Xdet_positive_decreasing}
Let $(b,B,m,\mu)$ be an admissible parameter set cf. Definition~\ref{def:admissibility}.
For $x\in \cH^+$ and $k\in \N$ let $\mathbf{x}^{(x,k)}$ be given by~\eqref{eq:solution-deterministic}.
Then  
\begin{equation}\label{eq:Xdet_pos}
  0\leq_{\cH^+} \mathbf{x}^{(x,k+1)}_t \leq_{\cH^+} \mathbf{x}_{t}^{(x,k)}
\end{equation}
for all $k\in \N$, $x\in \cH^+$, and $t\geq 0$.
\end{lemma}
\begin{proof}
It follows immediately from~Definition~\ref{def:admissibility} \ref{eq:linear-operator}
that $\cH \ni x \mapsto \tilde{b}^{(k)} + \tilde{B}^{(k)}(x) \in \cH$ is quasi-monotone 
with respect to $\cH^+$. As $\tilde{b}^{(k)}\in \cH^+$, Theorem~\ref{thm:comparison-theorem} with $K=\cH^+$,
$F(\cdot)=\tilde{b}^{(k)}+\tilde{B}^{(k)}(\cdot)$, $f\equiv 0$, and $g(\cdot)=\mathbf{x}_{\cdot}^{(x,k)}$ ensures that $\mathbf{x}^{(x,k)}_t\in \cH^+$
for all $t\geq 0$, $x\in \cH^+$, $k\in \N$.\par 
Moreover, for all $k\in \N$ and $x\in \cH^+$ we have 
\begin{align*}
\tilde{b}^{(k)} + \tilde{B}^{(k)}(x) - \left(\tilde{b}^{(k+1)} + \tilde{B}^{(k+1)}(x)\right) 
& \geq_{\cH^+} 0\,.
\end{align*}
This implies that for every $x\in \cH^+$, $k \in \mathbb{N}$, and $t\geq 0$ we have
\begin{align*}
\tfrac{\partial \mathbf{x}^{(x,k+1)}_{t} }{\partial t} - \left(\tilde{b}^{(k+1)} +   \tilde{B}^{(k+1)} (\mathbf{x}^{(x,k+1)}_t)\right) 
  &=\tfrac{\partial \mathbf{x}^{(x,k)}_{t} }{\partial t} - \left(\tilde{b}^{(k)} +   \tilde{B}^{(k)} (\mathbf{x}^{(x,k)}_t)\right)  \\
  &\leq_{\cH^+} \tfrac{\partial \mathbf{x}^{(x,k)}_{t}}{\partial t} - \left(\tilde{b}^{(k+1)} +   \tilde{B}^{(k+1)} (\mathbf{x}^{(x,k)}_t)\right)\,.
\end{align*}
Again applying Theorem \ref{thm:comparison-theorem} with $K = \mathcal{H}^+$, $F(\cdot)=\tilde{b}^{(k)} + \tilde{B}^{(k)}(\cdot)$, 
$f(t)=\mathbf{x}^{(x,k+1)}_t$ and $g(t)=\mathbf{x}^{(x,k)}_t$, $t\geq 0$, implies that $\mathbf{x}^{(x,k+1)}_t \leq_{\cH^+} \mathbf{x}^{(x,k)}_t$
for all $t\geq 0$.
\end{proof}
For $k\in \N$, $t\geq 0$ and $f\in\cB_{\rho}(\cH^{+}_{\textnormal{w}})$ define 
$P_t^{(\textnormal{det},k)}f \colon \cH^+ \rightarrow \R$ by
\begin{align}\label{semigroup-X-deterministic}
(P_{t}^{(\textnormal{det},k)}f)(x)\coloneqq f(\mathbf{x}^{(x,k)}_{t})\,,  \quad x\in \mathcal{H}^+.
\end{align}

\begin{lemma}\label{lem:Ptk_bdd}
Let $(b,B,m,\mu)$ be an admissible parameter set conform Definition~\ref{def:admissibility}.
Let $k\in \N$, $t\geq 0$, $f \in C_b(\cH^+_{\textnormal{w}})$ and let 
$P_{t}^{(\textnormal{det},k)}f\colon \cH^+\rightarrow \R$ be defined by~\eqref{semigroup-X-deterministic}. 
In addition, let 
\begin{align} 
M & \df 
\max\{ 
    1 
    + 
    2 \| \tilde{B}^{(1)} \|^{-2}_{\cL(\cH)} \| \tilde{b}^{(1)} \|^{2}
    , 2
    \}\,,\label{eq:M_for_Pkdet}
\\
\omega & \df 2 \| \tilde{B}^{(1)} \|_{\cL(\cH)}\,.\label{eq:omega_for_Pkdet}
\end{align}
Then $P_{t}^{(\textnormal{det},k)}f \in C_b(\cH^+_{\textnormal{w}})$, 
\begin{align}\label{eq:exponential_Pt_rho}
 \| P_{t}^{(\textnormal{det},k)}f \|_{\rho} 
 &\leq M \e^{\omega t} \| f \|_{\rho}\,,
\end{align}
and  
\begin{align}\label{eq:exponential_Pt_sqrtrho}
 \| P_{t}^{(\textnormal{det},k)}f \|_{\sqrt{\rho}} 
 &\leq \sqrt{M} \e^{\omega t/2} \| f \|_{\sqrt{\rho}}\,.
\end{align}
\end{lemma}

\begin{proof}
For every $t\geq 0$ the operator $\e^{t\tilde{B}^{(k)}}$ is strong-to-strong
continuous, hence it is also weak-to-weak continuous,
and thus $P_t^{(\textnormal{det},k)}f \in C_b(\cH^+_{\textnormal{w}})$.
Next note that
Lemma~\ref{lemma:Xdet_positive_decreasing} implies
that
\begin{equation}\label{eq:Xk_rho}
\begin{aligned}
 \tfrac{ 1 + \| \mathbf{x}^{(x,k)}_t \|^2 }{1+\| x \|^2 }
 & 
 \leq
 \tfrac{ 1 + \| \mathbf{x}^{(x,1)}_t \|^2 }{1+\| x \|^2 }
 \leq
 \tfrac{ 
    1 
    + 
    2 \e^{2 t \| \tilde{B}^{(1)} \|_{\cL(\cH)}} 
        ( 
            \| \tilde{B}^{(1)} \|^{-2}_{\cL(\cH)} 
            \| \tilde{b}^{(1)} \|^{2} 
            +  \| x \|^2
        ) 
 }{1+\| x \|^2 }
 \\ &
 \leq M \e^{\omega t}
\end{aligned}
\end{equation}
for all $x\in \cH^+$.
Using the above estimate and~\eqref{semigroup-X-deterministic} we obtain
\begin{align*}
\| P_t^{(\textnormal{det},k)}f\|_{\rho}
& =
\sup_{x\in \cH^+} 
  \tfrac{ (P_{t}^{(\textnormal{det},k)}f)(x)}{1 + \| x \|^2 } 
= 
\sup_{x\in \cH^+} 
  \tfrac{ f(\mathbf{x}^{(x,k)}_t) }{1 + \| x \|^2 } 
\leq
 \| f \|_{\rho}
 \sup_{x\in \cH^+} \tfrac{ 1 + \| \mathbf{x}^{(x,k)}_t \|^2 }{1+\| x \|^2 }
 \\ &\leq 
 M \e^{\omega t}
 \| f \|_{\rho}\,.
\end{align*}
Similarly, 
\begin{align*}
\| P_t^{(\textnormal{det},k)}f\|_{\sqrt{\rho}}
& =
\sup_{x\in \cH^+} 
  \tfrac{ f(\mathbf{x}^{(x,k)}_t) }{\sqrt{1 + \| x \|^2} } 
\leq
 \| f \|_{\sqrt{\rho}}
 \sup_{x\in \cH^+} \tfrac{ \sqrt{1 + \| \mathbf{x}^{(x,k)}_t \|^2 } }{\sqrt{ 1+\| x \|^2 }}
 \leq 
 \sqrt{M}\e^{\omega t/2}
 \| f \|_{\sqrt{\rho}}\,.
\end{align*}
\end{proof}
Recall that if $(A,\textnormal{dom}(A))$ is the generator 
of a strongly continuous semigroup $S=(S_t)_{t\geq 0}$ on a Banach space $X$,
then a subspace $D\subseteq \operatorname{dom}(A)$ is a \emph{core} for $A$ if 
$D$ is dense in $\operatorname{dom}(A)$ for the graph norm 
$\| \cdot \|_{\operatorname{dom}(A)} = \| \cdot \|_X + \| A \cdot \|_{X}$
(see~\cite[Chapter II, Def.\ 1.6]{EN00}). By~\cite[Chapter II, Prop.\ 1.7]{EN00} any subspace $D\subseteq \operatorname{dom}(A)$
that is dense in $X$ and invariant under $S$ is a core.
\begin{lemma}\label{lem:X-semigroup}
Let $(b,B,m,\mu)$ be an admissible parameter set conform Definition~\ref{def:admissibility}.
For all $k\in \N$, $t\geq 0$, $f\in\cB_{\rho}(\cH^{+}_{\textnormal{w}})$ let 
$P_{t}^{(\textnormal{det},k)}f\colon \cH^+\rightarrow \R$
be defined by~\eqref{semigroup-X-deterministic}. Then $(P_t^{(\textnormal{det},k)})_{t\geq 0}$ is a generalized Feller semigroup
on both $\cB_{\rho}(\cH^+_{\textnormal{w}})$ and
$\cB_{\sqrt{\rho}}(\cH^+_{\textnormal{w}})$ for all $k\in \N$.
Moreover $\cD$ is a core for the generator $\cG^{(k)}_{\textnormal{det}}$ of $(P_t^{(\textnormal{det},k)})_{t\geq 0}$ on $\cB_{\rho}(\cH^+_{\textnormal{w}})$ 
and for all $f\in \cD$ we have   
  \begin{align}\label{eq:Gk-det}
   (\cG^{(k)}_{\textnormal{det}} f)(x)=\langle \tilde{b}^{(k)}+\tilde{B}^{(k)}(x), f'(x)\rangle,\quad x\in \cH^+.  
  \end{align}
\end{lemma}
\begin{proof}
Let $k\in \N$. It follows from Lemma~\ref{lem:Ptk_bdd} that $(P_t^{(\textnormal{det},k)})_{t\geq 0}$ is a family of 
bounded linear operators on both $\cB_{\rho}(\cH^{+}_{\textnormal{w}})$ and $\cB_{\sqrt{\rho}}(\cH^{+}_{\textnormal{w}})$. 
Moreover, properties~\ref{semigroup1},~\ref{semigroup2}, and~\ref{semigroup5} in Definition~\ref{def:genFeller-semigroup} are trivially satisfied. Property~\ref{semigroup4} follows from Lemma~\ref{lem:Ptk_bdd}. Finally, property~\ref{semigroup3}
follows from Theorem~\ref{thm:charac_Brho} and the fact that
$\lim_{t\rightarrow 0^+} \| \mathbf{x}^{(x,k)}_t - x\| = 0$.\\
It is easily verified that $\cD$ is a subspace of $\cB_{\rho}(\cH_{\textnormal{w}}^+)$ that is invariant for 
$(P_t^{(\textnormal{det},k)})_{t\geq 0}$. 
We know from Lemma~\ref{lem:dense_subset_Brho} that $\cD$ is dense in $\cB_{\rho}(\cH^+_{\textnormal{w}})$, 
thus by~\cite[Chapter II, Prop.\ 1.7]{EN00} it remains to prove that
$\cD\subseteq \operatorname{dom}(\cG_{\textnormal{det}}^{(k)})$ and
that~\eqref{eq:Gk-det} holds. To this end, let $u\in\cHplus$ and
consider $f(\cdot)=\E^{-\langle u,\cdot\rangle}\in\cD$. For $f$ of this latter form, we define $f'(x) \df -\E^{-\langle u,x \rangle}u$, for $u, x \in \cHplus$ and 
$f''(x)$ to be the bounded linear map on $\cHplus$ defined for $u, x \in \cHplus$ by $f''(x)(v) \df \E^{-\langle u,x\rangle}u \langle u , v\rangle$, $v\in \cHplus$.
Now, observe that for
$\tilde{B}(x)\df\tilde{B}^{(k)}(x)+\tilde{b}^{(k)}$, we have
\begin{equation}\label{eq:dom_Gkdet}
 \begin{aligned}
 & \tfrac{(P_t^{(\textnormal{det},k)}f)(x) - f(x)}
 {t}
 -
 \langle f'(x), \tilde{B}(x) \rangle
 \\
 & =
    \int_{0}^{1}
        \left\langle 
            f'(s (\mathbf{x}_t^{(x,k)} -x ) + x) ,
            \frac{ \mathbf{x}_t^{(x,k)} - x }{t}  
            - \tilde{B}(x)
        \right\rangle 
    \D s
 \\ 
 & \quad 
 + \int_{0}^{1}\int_0^{1}
        \left\langle 
            f''\left(u s \big(\mathbf{x}_t^{(x,k)} - x\big) + x \right) 
            \left( s\big(\mathbf{x}_t^{(x,k)} -x\big)\right)
            ,
            \tilde{B}(x)
        \right\rangle 
    \D u \D s,
\end{aligned}
\end{equation}
where we used Lemma~\ref{lem:onesided_FTC} twice, which is applicable as the
one-sided derivatives of $f$, considered as a function on $\cHplus$,
exist. Observe that
\begin{equation}\label{eq:dom_Gkdet_lim1}
\begin{aligned}
 & 
 \lim_{t \to 0+} 
    \sup_{x\in \cH^+}
        \tfrac{ 
        \left| 
            \frac{1}{t}  \left( \mathbf{x}_t^{(x,k)} - x\right) 
            - \left(\tilde{B}^{(k)}x + \tilde{b}^{(k)}\right)
        \right|
        }
        { \sqrt{\rho(x)} }
\\ & \leq 
 \lim_{t \to 0+} 
    \sup_{x\in \cH^+}
        \tfrac{ 
            \| \tilde{B}^{(k)} \|_{\cL(\cH)} \|  e^{t \tilde{B}^{(k)}} - I \|_{\cL(\cH)} \| x \| 
            + 
            \frac{1}{t} \| \tilde{b}^{(k)} \| \int_{0}^{t} \| e^{(t-s) \tilde{B}^{(k)}}  - I \|_{\cL(\cH)} \D s
        }
        { \sqrt{ 1 + \| x \|^2 }}
 =0
 .
 \end{aligned}
\end{equation}
Moreover we have 
\begin{equation}\label{eq:dom_Gkdet_lim2}
\begin{aligned}
 \lim_{t \to 0+} 
    \sup_{x\in \cH^+}
         \tfrac{ | \mathbf{x}_t^{(x,k)} -x |}{\sqrt{\rho(x)}}
 \leq 
 \lim_{t \to 0+} 
    \sup_{x\in \cH^+}
         \tfrac{ \| e^{t \tilde{B}^{(k)}}  - I \|_{\cL(\cH)} \| x \| 
         + \int_{0}^{t} \| e^{(t-s) \tilde{B}^{(k)}} \tilde{b}^{(k)}\| \D s
         }
         {\sqrt{\rho(x)}}
 =0.
\end{aligned}
\end{equation}
Since
$\sup_{x\in \cH^+} |\rho(x)|^{-\frac{1}{2}} \| f'(x) \| < \infty$
and 
$\sup_{x\in \cH^+} \| f''(x) \|_{\cL(\cH)} < \infty$, it follows from
equations~\eqref{eq:dom_Gkdet},~\eqref{eq:dom_Gkdet_lim1},
and~\eqref{eq:dom_Gkdet_lim2} that
\begin{equation}
 \lim_{t\to 0+ }
 \left\| 
    \tfrac{(P_t^{(\textnormal{det},k)}f)(x) - f(x)}
    {t}
    -
    \langle f'(x), \tilde{B}^{(k)}(x) + \tilde{b}^{(k)} \rangle
 \right\|_{\rho}
 = 0.
\end{equation}
This, the linearity of $\cG_{\textnormal{det}}^{(k)}$ and the fact that $\cD$
is invariant for $P_t^{(\textnormal{det},k)}$ (and thus
$P_t^{(\textnormal{det},k)}f \in \cB_{\rho}(\cH^+_{\operatorname{w}})$ whenever $f\in \cD$)
implies that $\cD\subseteq \operatorname{dom}(\cG_{\textnormal{det}}^{(k)})$
and that~\eqref{eq:Gk-det} holds.
\end{proof}
We now introduce the family of measures $\nu^{(k)}\colon \cH^+ \times
\cB(\cHplus\setminus\set{0})\to [0,\infty)$ for every $x\in\cHplus$ given by 
\begin{equation}\label{eq:def_nuk}
\nu^{(k)}(x, \D \xi) = \dmk + \frac{\langle \dmuk, x\rangle}{\|\xi\|^2}
\end{equation}
and define the operator 
$\cG^{(k)}_{\textnormal{jump}} \colon \operatorname{dom}(\cG^{(k)}_{\textnormal{jump}}) \subseteq  \cB_{\rho}(\cH^+_{\textnormal{w}}) 
\rightarrow \cB_{\rho}(\cH^+_{\textnormal{w}})$ 
by 
\begin{align}\label{eq:domDGjump}
&\operatorname{dom}(\cG^{(k)}_{\textnormal{jump}}) \notag
\\ &=
\left\{ f \in \cB_{\rho}(\cH^+_{\textnormal{w}}) \colon 
    \left( 
        x\mapsto \int_{\cHpluso}\left(f(\xi+x)-f(x)\right)\nu^{(k)}(x,\D\xi)
    \right)
    \in 
    \cB_{\rho}(\cH^+_{\textnormal{w}})
\right\}\,\,  
\end{align}
and for $f\in \operatorname{dom}(\cG^{(k)}_{\textnormal{jump}})$:
\begin{align}\label{eq:gen-jump-k}
\cG^{(k)}_{\textnormal{jump}} f(x)\coloneqq \int_{\cHpluso}\left(f(\xi+x)-f(x)\right)\nu^{(k)}(x,\D\xi),
\quad 
x\in \cH^+.
\end{align}
 Note that for all $k\in\MN$ the measure $\nu^{(k)}(x,\D\xi)$ is finite,
i.e. $\nu^{(k)}(x,\cHpluso)<\infty$ for all $x\in\cHplus$, but it is an affine
function in $x$ and hence unbounded in the first component. For that reason
$\cG^{(k)}_{\textnormal{jump}} f$ may not be in
$\cB_{\rho}(\cHplus_{\textnormal{w}})$ for all $f\in
\cB_{\rho}(\cHplus_{\textnormal{w}})$. However, the following lemma ensures that $C_b(\cH^+_{\textnormal{w}})\subseteq \operatorname{dom}(\cG^{(k)}_{\textnormal{jump}})$:
\begin{lemma}\label{lem:G-jump} 
Let $(b,B,m,\mu)$ be an admissible parameter set conform Definition~\ref{def:admissibility}.
Let $k\in \N$, and let $\cG^{(k)}_{\textnormal{jump}}$ be as defined in~\eqref{eq:domDGjump}
and~\eqref{eq:gen-jump-k}.
Then $C_b(\cH^+_{\textnormal{w}})\subseteq \operatorname{dom}(\cG^{(k)}_{\textnormal{jump}})$. 
\end{lemma}
\begin{proof}
Let $f\in  C_b(\cH^+_{\textnormal{w}})$
and let $g_f\colon \cH^+ \rightarrow \R$ be defined by 
\begin{equation}
g_f(x) 
=  
\int_{\cH^+\setminus\{0\}} 
    f(x+\xi) \tfrac{\langle \mu^{(k)}(\D\xi), x \rangle}{\| \xi \|^2 } 
\end{equation}

We will prove that $g_f\in \cB_{\rho}(\cH^+_{\textnormal{w}})$ using Theorem~\ref{thm:charac_Brho}. All other terms 
in the definition of $\cG_{\textnormal{jump}}^{(k)}f$ can be 
dealt with in a similar (simpler) way.\par 
To see that $g_f$ is continuous on $K_R := \{\rho \leq R \}$ for all $R> 0$ it suffices to show that $g_f$ is sequentially continuous on $K_R$ for every $R> 0$ as the weak topology restricted to $K_R$ is metrizable. Fix $R> 0$ and let $(x_n)_{n\in \N}$ be a sequence in $K_R$ converging (weakly) to an $x\in K_R$. By the dominated convergence theorem (Theorem~\ref{thm:DCT}) and the 
fact that $\sup_{n\in \N} \|x_n \| \leq \sqrt{R}$ we obtain 
\begin{align*}
 \lim_{n\rightarrow \infty}
 \left|
    g_f(x_n)
    -
    g_f(x)
 \right|
 &\leq \lim_{n\rightarrow \infty}
 \left\|
    \int_{\cH^+ \setminus \{ 0 \}}
        (f(x_n + \xi) - f(x+\xi))
    \tfrac{\mu^{(k)}(\D \xi)}{\| \xi\|^2}
 \right\| \| x_n \|
 \\ & \quad 
 +
 \lim_{n\rightarrow \infty} \left|
    \int_{\cH^+ \setminus \{ 0 \}}
        f(x+\xi)
    \tfrac{\langle \mu^{(k)}, x_n - x \rangle (\D \xi)}{\| \xi\|^2}
 \right|=0.
\end{align*}
Finally, observe that 
$\lim_{R\rightarrow \infty} \sup_{x\in \cH^+: \rho(x)\geq R}
|\rho(x)|^{-1}|g_f(x)| =0$
as $f$ is bounded and 
$\int_{\cH^+\setminus\{0\}} \frac{\mu^{(k)}(\D \xi)}{\| \xi \|^2} \in \cH$ (recall~\eqref{eq:normintest}). By Theorem~\ref{thm:charac_Brho} this ensures that $g_f\in \cB_{\rho}(\cH^+_{\textnormal{w}})$, which completes the proof of the lemma.
\end{proof}

%
      %
In the next proposition we achieve an important intermediate stage, that
allows us to conclude the existence of generalized Feller processes in $\cHplus$ admitting
for bounded drifts and finite-activity jump behavior, as well as satisfying the
exponential affine formula \eqref{eq:affine-formula-intro}:
\begin{proposition}\label{prop:Gk_generator}
Let $(b,B,m,\mu)$ be an admissible parameter set conform Definition~\ref{def:admissibility}.
Let $k\in \N$, and let $(\phi^{(k)}(\cdot,u),\psi^{(k)}(\cdot,u))$ be the unique solution to \eqref{eq:riccati-conc-k} (cf.\ Proposition~\ref{prop:global_sol_finact}). 
Let $\cD \subseteq \cB_{\rho}(\cH^+_{\textnormal{w}})$ be given by~\eqref{eq:span-Fourier}
and $\cG^{(k)}_{\textnormal{det}}$ and $\cG^{(k)}_{\textnormal{jump}}$ be as
defined in~\eqref{eq:Gk-det}, respectively ~\eqref{eq:gen-jump-k}. Consider the operator $\cG^{(k)}_\textnormal{det}+\cG^{(k)}_{\textnormal{jump}}\colon
\operatorname{dom}(\cG^{(k)}_{\textnormal{det}}) \cap \operatorname{dom}(\cG^{(k)}_{\textnormal{jump}}) \subseteq
\cB_{\rho}(\cH^+_{\textnormal{w}})\rightarrow 
\cB_{\rho}(\cH^+_{\textnormal{w}})$.
Then
$\cD \subseteq 
\operatorname{dom}(\cG^{(k)}_{\textnormal{det}}) \cap \operatorname{dom}(\cG^{(k)}_{\textnormal{jump}})$.
Moreover, there exists a generalized Feller semigroup $(P_t^{(k)})_{t\geq 0}$ with generator $(\cG^{(k)}, \operatorname{dom}(\cG^{(k)}))$ such that 
\begin{enumerate}
\item\label{it:Gk_generator:domain} 
$\cD \subseteq \operatorname{dom}(\cG^{(k)})$,
\item\label{it:Gk_generator:sum} $\cG^{(k)}f =(\cG^{(k)}_\textnormal{det}+\cG^{(k)}_{\textnormal{jump}})f$ for all $f\in \cD$,
\item\label{it:Gk_generator:Markovian} $P^{(k)}_t 1 = 1$ for all $t\geq 0$, 
and 
\item\label{it:Gk_generator:affine} for all $u,x\in \cH^{+}$, $t\geq 0$ we have 
\begin{align}\label{eq:affine-X-tilde-k}
\left( P^{(k)}_t\E^{-\langle \cdot,u \rangle}\right) (x)
= \E^{-\phi^{(k)}(t,u)-\langle x,\psi^{(k)}(t,u)\rangle}.
\end{align}

\end{enumerate}
\end{proposition}
\begin{proof}[Proof of Proposition~\ref{prop:Gk_generator}]
Roughly speaking, we can ensure the existence of a generalized Feller semigroup $P_t^{(k)}$ satisfying~\ref{it:Gk_generator:sum} in Proposition~\ref{prop:Gk_generator} by verifying that all conditions of~\cite[Proposition 3.3]{CT20} are satisfied. However, the assertions of~\cite[Proposition 3.3]{CT20} do not immediately give us~\ref{it:Gk_generator:domain},~\ref{it:Gk_generator:Markovian}, and~\ref{it:Gk_generator:affine}. In order to obtain these statements we need to dig into the proof of~\cite[Proposition 3.3]{CT20}, which makes this proof somewhat technical and tricky. To enhance the readability, we split the proof in to several parts.\par
\par\medskip 
{\it Step 1: Verifying the assumptions of~\cite[Proposition 3.3]{CT20}.}
We consider, in the notation of that Proposition, $(X,\rho)=(\cH^+_{\textnormal{w}},\rho)$,
$A= \cG^{(k)}_{\textnormal{det}}$, 
$\omega$ as in~\eqref{eq:omega_for_Pkdet}, $M_1 = M$ where $M$ is as in~\eqref{eq:M_for_Pkdet},
$\mu(x,E) = \nu^{(k)}(x,E-x\cap \cH^+)$ (recall the definition
of $\nu^{(k)}$ from~\eqref{eq:def_nuk}; here $E-x \df \{ y\in \cH \colon y+x\in E \}$),
and $B=\cG^{(k)}_{\textnormal{jump}}$.
By Lemma~\ref{lem:X-semigroup}, $\cG^{(k)}_{\textnormal{det}}$ is the generator of a generalized 
Feller semigroup $(P^{(\textnormal{det},k)}_t)_{t\geq 0}$ of transport type on both $\cB_{\rho}(\cH^+_{\textnormal{w}})$ and $\cB_{\sqrt{\rho}}(\cH^+_{\textnormal{w}})$.
In particular, by~\cite[Theorem 3.2]{DT10}, $(P^{(\textnormal{det},k)}_t)_{t\geq 0}$ defines
a strongly continuous semigroup on both $\cB_{\rho}(\cH^+_{\textnormal{w}})$ and $\cB_{\sqrt{\rho}}(\cH^+_{\textnormal{w}})$,
i.e., it automatically holds that the domain of $\cG^{(k)}_{\textnormal{det}}$ is dense and that $P^{(\textnormal{det},k)}_t)_{t\geq 0}$ allows for exponential bounds (see Lemma~\ref{lem:Ptk_bdd} for explicit bounds). 
Lemma \ref{lem:G-jump} implies that $\cG^{(k)}_{\textnormal{jump}}f$ is weakly
continuous on compact sets $\{ \rho \leq R \}$ for all $R\geq 0$ and all $f\in C_b(\cH^+_{\textnormal{w}})$.\par 
Moreover, one easily verifies that there exists a constant $K$ (possibly depending on $k$)
such that for all $x\in \cH^+$ we have
\begin{equation}
\begin{aligned}
 \int_{\cH^+\setminus \{0\}} \rho(y+x) \,\nu^{(k)}(x,\D y)
 \leq
  \int_{\cH^+\setminus \{0\}} ( 1 + 2\| x \|^2 + 2\|y \|^2) \,\nu^{(k)}(x,\D y)
 \leq 
 K |\rho(x)|^2\,, 
\end{aligned}
\end{equation}
\begin{equation}
\begin{aligned}
 \int_{\cH^+\setminus \{0\}} \sqrt{\rho(y+x)} \,\nu^{(k)}(x,\D y)
 \leq
  \int_{\cH^+\setminus \{0\}} ( 1 + \| x \| + \|y \|) \,\nu^{(k)}(x,\D y)
 \leq 
 K \rho(x)\,, 
\end{aligned} 
\end{equation}
and 
\begin{equation}\label{eq:CT_3.3_cond3}
 \int_{\cH^+\setminus \{0\}} \nu^{(k)}(x,\D y) \leq K \sqrt{\rho(x)}\,.
\end{equation}
Next, observe that by Lemma~\ref{lemma:Xdet_positive_decreasing} and the fact that $(0,B,0,\mu)$ is also an admissible parameter set, we have $e^{t\tilde{B}^{(k)}}\xi\in \cH^+$
whenever $\xi\in \cH^{+}$. Thus 
\begin{equation}
\begin{aligned}
 & P_t^{(\textnormal{det},k)} \rho(\xi + x)
 = 1 + \| \mathbf{x}^{(\xi+x,k)}_t \|^2
 = 1 + \| e^{t\tilde{B}^{(k)}}\xi+ \mathbf{x}^{(\xi,k)}_t \|^2 \\ &
 = P_t^{(\textnormal{det},k)} \rho(x) 
 + 2\langle e^{t\tilde{B}^{(k)}}\xi, \mathbf{x}^{(\xi,k)}_t\rangle
 + \| \e^{t\tilde{B}^{(k)}}\xi  \|^2
 \geq P_t^{(\textnormal{det},k)} \rho(x)
 \end{aligned}
\end{equation}
for all $x,\xi \in \cH^+$. This together with estimates similar to~\eqref{eq:Xk_rho} yields (note that $\|\e^{t\tilde{B}^{(k)}}\xi \| \leq \|\e^{t\tilde{B}^{(1)}}\xi \|$, and recall $\omega$ from~\eqref{eq:omega_for_Pkdet})
\begin{equation}
\begin{aligned}
&\left|
  \frac{
    \sup_{t\geq 0} \e^{-\omega t} P_t^{(\textnormal{det},k)} \rho(\xi + x) - 
    \sup_{t\geq 0} \e^{-\omega t} P_t^{(\textnormal{det},k)} \rho(x)
  }{
    \sup_{t\geq 0} \e^{-\omega t} P_t^{(\textnormal{det},k)} \rho(x)
  }
\right|
\\ & =
  \frac{
    \sup_{t\geq 0} \e^{-\omega t} P_t^{(\textnormal{det},k)} \rho(\xi + x) - 
    \sup_{t\geq 0} \e^{-\omega t} P_t^{(\textnormal{det},k)} \rho(x)
  }{
    \sup_{t\geq 0} \e^{-\omega t} P_t^{(\textnormal{det},k)} \rho(x)
  } 
\\ & \leq 
  \frac{
    \sup_{t\geq 0} \e^{-\omega t} \left|  
        \| \e^{t\tilde{B}^{(k)}}\xi  \|^2
        + 
        2\| \e^{t\tilde{B}^{(k)}}\xi\| \| \mathbf{x}^{(x,k)}_t \| 
    \right|
  }{
    1 + \| x \|^2
  }
\leq  
  \frac{
    \| \xi \|^2 + 2\| \xi \| ( \| x \| + \sqrt{M})
  }{
    1 + \| x \|^2
  }
\\ & 
\leq  
  \frac{
    (M+2\| \xi \|^2 ) (1+\|x\|)  
  }{
    1 + \| x \|^2
  }
\leq 
\tfrac{2M + 4\| \xi \|^2}
{ 1 + \| x \| }
\end{aligned}
\end{equation}
for all $x,\xi\in \cH^+$.
It follows that for all $x\in \cH^+$ we have
\begin{equation}
\begin{aligned}\label{eq:tilde_omega_est}
 & 
 \int_{\cH^+\setminus \{0\}}
 \left| 
    \frac{ \sup_{t\geq 0} 
        \e^{-\omega t} \left( 
        P_t^{(\textnormal{det},k)}\rho
        \right)(\xi+x)
        -
        \sup_{t\geq 0}
        \e^{-\omega t} \left( 
        P_t^{(\textnormal{det},k)}\rho
        \right)(x)
    }{
        \sup_{t\geq 0}
        \e^{-\omega t} \left( 
        P_t^{(\textnormal{det},k)}\rho
        \right)(x)
    }
 \right|\,\nu^{(k)}(x,\D \xi)
 \\ & 
 \leq \sup_{y\in \cH^+}\left| 
 \int_{\cH^+\setminus \{0\}}
 \left(      
    \tfrac{2M + 4\| \xi \|^2}
    { 1 + \| y \| }
 \right)\,\nu^{(k)}(y,\D \xi) \right|
 =: \tilde{\omega}_k < \infty.
\end{aligned}
\end{equation}
This ensures
that all conditions of~\cite[Proposition 3.3]{CT20}
are satisfied. 

{\it Step 2: Presenting the assertions of~\cite[Proposition 3.3]{CT20}.}
As in the proof of~\cite[Proposition 3.3]{CT20}, we introduce the operator
$\cG^{(k,n)}_{\textnormal{jump}} \in \cL(\cB_{\rho}(\cH^+_{\textnormal{w}}))$
which satisfies
$$ 
(\cG^{(k,n)}_{\textnormal{jump}} f)(x) 
= \int_{\cH^+\setminus \{0\}} 
    (f(\xi+x)-f(x))\tfrac{n}{\rho(\xi+x)\wedge n}
    \,\nu^{(k)}(x,\D \xi)\,
$$
for all $x\in \cH^+$, $f\in \cB_{\rho}(\cH^+_{\textnormal{w}})$.
Note that $\cD\subseteq \operatorname{dom}(\cG^{(k)}_{\textnormal{jump}})$ by Lemma~\ref{lem:G-jump}.
For future reference (see Proposition~\ref{prop:cadlag} below) we also introduce $\tilde{\rho}_k\colon \cH^+ \rightarrow \R$, $\tilde{\rho}_k(x) = \sup_{t\geq 0} e^{-\omega t} P_t^{(\textnormal{det},k)} \rho(x)$. It follows from~\cite[Remark 2.9]{CT20} that $\tilde{\rho}_k$
 is an admissible weight function and that $\left\| \cdot \right\|_{\rho} \leq \left\| \cdot \right\|_{\tilde{\rho}_k} \leq M \left\| \cdot \right\|_{\rho}$. Moreover,
it follows from the proof
of~\cite[Proposition 3.3]{CT20} (with $A=\cG^{(k)}_{\textnormal{det}}$ and $B_n=\cG^{(k,n)}_{\textnormal{jump}}$) that 
$\cG^{(k)}_{\textnormal{det}} + \cG^{(k,n)}_{\textnormal{jump}}$
is the generator of a generalized Feller semigroup 
$(P^{(k,n)}_t)_{t\geq 0}$ on $\cB_{\rho}(\cH^+_{\textnormal{w}})$
for all $n\in \N$, such that 
\begin{enumerate}
 \item [a)]\label{it:quasi-contractive} $ \| P_t^{(k,n)}\|_{\cL(\cB_{\tilde{\rho}_k}(\cH^+_{\textnormal{w}}))} 
  \leq e^{(\omega+\tilde{\omega}_k)t}$ for all $t\geq 0$, $n\in \N$,
 \item [b)]\label{it:PknUnifbdd} $\| P^{(k,n)}_t \|_{\cL(\cB_{\rho}(\cH^+_{\textnormal{w}}))}
 \leq M \e^{(\omega+\tilde{\omega}_k)t}$ for all $t\geq 0$, $n\in \N$,
 \item [c)] \label{it:AknCauchy} $\lim_{n\rightarrow \infty}
 \|( \cG^{(k,n)}_{\textnormal{jump}} - \cG^{(k)}_{\textnormal{jump}})f \|_{\rho} = 0 $ for all $f\in \cD$\,.
 \end{enumerate}
It moreover follows from the proof of~\cite[Proposition 3.3]{CT20} that there exists a generalized Feller semigroup 
$(P^{(k)}_t)_{t\geq 0}$ on $\cB_{\rho}(\cH^+_{\textnormal{w}})$ 
with generator $\cG^{(k)}$ satisfying
\begin{equation}\label{it:Pknconv} 
\lim_{n\rightarrow \infty}\sup_{s\in [0,t]}
 \| (P_s^{(k,n)} - P_s^{(k)} )f \|_{\rho} = 0, \,\, \mbox{for all}\,\, f \in \cB_{\rho}(\cH^+_{\textnormal{w}}),\,\, t\geq 0.
\end{equation} 
\medskip 
\emph{Step 3: Proof of~\ref{it:Gk_generator:domain} and~\ref{it:Gk_generator:sum}.}
Fix $f\in \cD$. Let $u_{k,n}(t) = P^{(k,n)}_t f$, $t\geq 0$ and $n\in \N$, let $u_k(t)= P^{(k)}_t f$, $t\geq 0$, and let $v_k(t)=P^{(k)}_t (\cG^{(k)}_{\textnormal{det}} + \cG^{(k)}_{\textnormal{jump}})f$. Observe that $u_{k,n}'(t)= P^{(k,n)}_t (\cG^{(k)}_{\textnormal{det}} + \cG^{(k,n)}_{\textnormal{jump}}) f$. By a), b),
and~\eqref{it:Pknconv} we have, for all $T\geq 0$, that
\begin{equation}
\lim_{n\rightarrow \infty} \sup_{t\in [0,T]} 
\left( \| u_{k,n}(t) - u_k(t)\|_{\rho} + \| u_{k,n}'(t) - v_k(t)\|_{\rho}\right) = 0. 
\end{equation}
This implies that $u_k$ is differentiable and $u_k'(t) = v_k(t)$, which implies that $f\in \operatorname{dom}(\cG^{(k)})$ and $\cG^{(k)}f=u_k'(0)=(\cG^{(k)}_{\textnormal{det}} + \cG^{(k)}_{\textnormal{jump}})f$. 
\par\medskip
\emph{Step 4: Proof of~\ref{it:Gk_generator:Markovian}.}
In order to verify
that $P_t^{(k)}1=1$ for all $t\geq 0$, observe that 
$\cG^{(k,n)}_{\textnormal{jump}} 1 = 0$ (whence $e^{t \cG^{(k,n)}_{\textnormal{jump}}} 1 = 1$ for all $t\geq 0$), whence 
the Trotter product formula (see, e.g.,~\cite[Chapter III, Corollary
5.8]{EN00}) implies that 
$P^{(k,n)}_t 1 = 1$ for all $t\geq 0$. It follows that
$P^{(k)}_t 1 = 1 $ for all $t\geq 0$.\par 

\par\medskip
\emph{Step 5: Proof of~\ref{it:Gk_generator:affine}.} Recall the definition of $R^{(k)}$ and $F^{(k)}$ from~\eqref{eq:Fk} and~\eqref{eq:Rk}. 
Recall from Lemmas~\ref{lem:X-semigroup} and~\ref{lem:G-jump} that $\e^{-\langle\cdot,u\rangle}\in \cD \subseteq \operatorname{dom}(\cG^{(k)}_{\textnormal{det}})\cap \operatorname{dom}(\cG^{(k)}_{\textnormal{jump}})$ for all $u\in \cH^+$, and that   
\begin{align}\label{eq:Kolomgorov-Abs-k}
\cG^{(k)}(\e^{-\langle\cdot,u\rangle})(x)
& =
(\cG^{(k)}_{\textnormal{det}}+\cG^{(k)}_{\textnormal{jump}})(\e^{-\langle\cdot,u\rangle})(x) \notag
\\ &=\Big(-\langle
  \tilde{b}^{(k)}+\tilde{B}^{(k)}(x),u\rangle+\int\limits_{\cHplus\setminus
  \{0\}}\big(\E^{-\langle \xi, u\rangle}-1\big) \nu^{(k)}(x,\D\xi)\Big)\E^{-\langle x,
  u\rangle}\nonumber \\
  &= \big(-F^{(k)}(u)-\langle x,R^{(k)}(u)\rangle\big) \E^{-\langle x , u\rangle}
\end{align}
for all $u,x\in \cH^+$. On the other hand, Proposition~\ref{prop:global_sol_finact}
implies that
\begin{align*}
&\frac{\partial }{\partial  t}\E^{-\phi^{(k)}(t,u)-\langle x, \psi^{(k)}(t,u)\rangle}
\\ & =\big(-F^{(k)}(\psi^{(k)}(t,u))
-\langle x,R^{(k)}(\psi^{(k)}(t,u))\rangle\big)
\E^{-\phi^{(k)}(t,u)-\langle x, \psi^{(k)}(t,u)\rangle}
\end{align*}
for all $u,x\in \cH^+$.
Therefore for all $u\in \cH^+$ it holds that the function 
$[0,\infty)\ni t \mapsto \E^{-\phi^{(k)}(t,u)-\langle \cdot , \psi^{(k)}(t,u)\rangle} \in \cD\subseteq \operatorname{dom}(\cG^{(k)})$ is a classical solution to the following abstract Cauchy problem:
\begin{align*}
\left\{ 
\begin{aligned}
 \frac{\partial}{\partial t}v(t)
 &=
 \cG^{(k)} v(t),\\
  v(0)&=\E^{-\langle \cdot, u\rangle}.
\end{aligned}\right.
\end{align*}
By the uniqueness of the classical solution we conclude~\eqref{eq:affine-X-tilde-k}.
\end{proof}

From Proposition~\ref{prop:Gk_generator} on the existence of the generalized
Feller semigroup $(P_{t}^{(k)})$ with $P_{t}^{(k)}1=1$, together with
the version of Kolmogorov's extension Theorem~\ref{thm:Kol-extension-theorem},
we conclude that there exists a generalized Feller process associated to
$(P_{t}^{(k)})_{t\geq 0}$, denoted by $(X_{t}^{(k)})_{t\geq 0}$, such that
$\EXspec{x}{f(X_{t}^{(k)})}=P_{t}^{(k)}f(x)$ for every $f\in \cB_{\rho}(\cHplus_{\textnormal{w}})$.
Item~a) and equation~\eqref{it:Pknconv} in the proof of
Proposition~\ref{prop:Gk_generator} result in exponential bounds on $\|
P_t^{(k)} \|_{\cL(\cB_{\rho}(\cH^+_{\textnormal{w}}))}$ \emph{that depend on
  $k\in \N$}. In order to proceed, we need to establish bounds that are
\emph{uniform} in $k$. We begin with a lemma that builds on top of the
results in Proposition \ref{prop:gateaux-diff}:
\begin{lemma}\label{lem:uniform-bound-k}
Let $(b,B,m,\mu)$ be an admissible parameter set conform
Definition~\ref{def:admissibility}. Moreover for every $k\in\MN$, let
$(\phi^{(k)}(\cdot,u),\psi^{(k)}(\cdot,u))$ be the solution of
\eqref{eq:riccati-conc-k}, the existence of which is established in Proposition
\ref{prop:global_sol_finact}, and the mappings $\D_{+}\phi(\cdot,0)$,
$\D_{+}\psi(\cdot,0)$, $\D_{+}^{2}\phi^{(k)}(\cdot,0)$ and
$\D_{+}^{2}\psi^{(k)}(\cdot,0)$ be as in Proposition~\ref{prop:gateaux-diff}
for the admissible parameter set $(b,B,m^{(k)},\mu^{(k)})$. Moreover, let $(X_{t}^{(k)})_{t\geq 0}$ be the generalized
Feller process associated to $(P_{t}^{(k)})_{t\geq 0}$. Then for every $v,w\in\cH$ and $t\geq 0$ the following formulas hold true:
 \begin{align}\label{eq:explicit-Xv-k}
  \EXspec{x}{\langle X^{(k)}_{t}, v\rangle}&=\D_{+}\phi(t,0)(v)+\langle x, \D_{+}\psi(t,0)(v)\rangle,
 \end{align}
 and
 \begin{align}\label{eq:explicit-Xv-square-k}
  \EXspec{x}{\langle X^{(k)}_{t}, v\rangle\langle X^{(k)}_{t},w\rangle}&=-\D^{2}_{+}\phi^{(k)}(t,0)(v,w)-\langle x,
                             \D^{2}_{+}\psi^{(k)}(t,0)(v,w)\rangle \nonumber\\
  &\quad+\big(\D_{+}\phi(t,0)(v)+\langle x,
    \D_{+}\psi(t,0)(v)\rangle\big)\nonumber\\
   &\quad\times \big(\D_{+}\phi(t,0)(w)+\langle x, \D_{+}\psi(t,0)(w)\rangle\big).
 \end{align}
\end{lemma}
\begin{proof}
Let $k\in\MN$ arbitrary, but fixed. Recall from
Remark~\ref{rem:semigroup-on-b-rho} that for all $t\geq 0$:
\begin{align}\label{eq:second-moment-finite-2}
\EXspec{x}{\norm{X^{(k)}_{t}}^{2}}<\infty,\quad \forall x\in\cHplus.  
\end{align}
 We first show that the formulas
\eqref{eq:explicit-Xv-k} and \eqref{eq:explicit-Xv-square-k} holds for
$v,w\in\cHplus$ and subsequently extend these to $v,w\in \cH$.
Let $u\in\cHplus$, $x\in\cHplus$ and $t\geq 0$, then we set 
\begin{align*}
 \Phi^{(k)}(t,u,x)\df\E^{-\phi^{(k)}(t, u)-\langle x, \psi^{(k)}(t, u)\rangle}, 
\end{align*}
and by the affine property of $(X_{t}^{(k)})_{t\geq 0}$ from equation \eqref{eq:affine-X-tilde-k} we have
\begin{align}\label{eq:expo-affine-theta-k}
\EXspec{x}{\E^{-\langle X^{(k)}_{t}, u\rangle}}=\Phi^{(k)}(t,u,x).
\end{align}
By Proposition \ref{prop:gateaux-diff} the right-hand side of equation
\eqref{eq:expo-affine-theta-k} is one-sided differentiable in $u\in\cHplus$ in
the direction $v$ for every $v\in\cHplus$. In particular, by applying the chain-rule at $u=0$ we have:
\begin{align}
  \D_{+}\Phi^{(k)}(t,0,x)(v)&
                                                                     =\big(-\D_{+}\phi^{(k)}(t,0)(v)-\langle x,\D_{+}\psi^{(k)}(t,0)(v)\rangle
                                    \big)\Phi^{(k)}(t,0,x) \nonumber\\ 
  &=-\D_{+}\phi^{(k)}(t,0)(v)-\langle x, \D_{+}\psi^{(k)}(t,0)(v)\rangle,\label{eq:explicit-Xv-right-k-1}  
\end{align}
where $\D_{+}\phi^{(k)}(t,0)=\D_{+}\phi(t,0)$ and $\D_{+}\psi^{(k)}(t,0)=\D_{+}\psi(t,0)$ for all $t\geq 0$ and $k\in\MN$, see Lemma~\ref{lem:R-one-sided-dif}.
Moreover, note that for $\theta\in\MRplus$ the random variable $\E^{-\langle
  X^{(k)}_{t},\theta v\rangle}$ is integrable and for $\MP_{x}$-almost all $\omega\in\Omega$ the
mapping $\theta\mapsto \E^{-\langle X^{(k)}_{t}(\omega),\theta v\rangle}$  
is differentiable. Due to equation~\eqref{eq:second-moment-finite-2} the term
\begin{align*}
 \sup\limits_{\theta \in [0,1]}\lvert\frac{\D}{\D \theta}\E^{-\langle X^{(k)}_{t},
  \theta v\rangle}\rvert= \sup\limits_{\theta \in [0,1]} \lvert-\langle
  X^{(k)}_{t}, v\rangle\E^{-\langle X^{(k)}_{t}, \theta v\rangle}\rvert
\end{align*}
is integrable. 
 Hence, all the
requirements for switching the derivative with respect to $\theta$ and the expectation with respect to
$\MP_{x}$ are fulfilled, thus the left-hand side of equation
\eqref{eq:expo-affine-theta-k} together with equation
\eqref{eq:explicit-Xv-right-k-1} yields:
\begin{align}\label{eq:explicit-Xv-right-2-k}
\EXspec{x}{\langle X^{(k)}_{t},v\rangle}=\D_{+}\phi(t,0)(v)+\langle x, \D_{+}\psi(t,0)(v)\rangle.
\end{align}
Again due to equation~\eqref{eq:second-moment-finite-2} we obtain by
differentiating both sides of equation \eqref{eq:expo-affine-theta-k} at $u=0$ twice in the direction $v$ and $w$ the
formula in \eqref{eq:explicit-Xv-square-k}.
Note that for every $v\in\cH$ there exist
$v^{+},v^{-}\in\cHplus$ such that $v=v^{+}-v^{-}$, by linearity of the formula
\eqref{eq:explicit-Xv-k} in $v$, we have:
\begin{align*}
 \EXspec{x}{\langle X^{(k)}_{t}, v\rangle}&= \EXspec{x}{\langle X^{(k)}_{t},
                                 v^{+}\rangle}-\EXspec{x}{\langle X^{(k)}_{t},
                                 v^{-}\rangle}\\
                               &=\D_{+}\phi(t,0)(v^{+})-\D_{+}\phi(t,0)(v^{-})\\
  &\quad +\langle x,
    \D_{+}\psi(t,0)(v^{+})-\D_{+}\psi(t,0)(v^{-})\rangle\\
  &=\D_{+}\phi(t,0)(v)+\langle x, \D_{+}\psi(t,0)(v)\rangle.
\end{align*} 
By introducing the linear functional
\begin{align}\label{eq:linear-tensor}
\llangle \cdot, v\otimes w\rrangle\colon \cH\otimes\cH\to \MR\text{ defined by
  }\llangle x\otimes x, v\otimes w \rrangle\df\langle x, v\rangle\langle x, w\rangle,
\end{align}
we can write $\EXspec{x}{\langle X_{t}^{(k)}, v\rangle\langle X_{t}^{(k)},
w\rangle}=\EXspec{x}{\llangle X_{t}^{(k)}\otimes X_{t}^{(k)}, v\otimes w
\rrangle}$ for every $v,w\in\cHplus$ and we have
\begin{align}\label{eq:explicit-Xv-square-k-tensor}
\EXspec{x}{\llangle X_{t}^{(k)}\otimes X_{t}^{(k)}, v\otimes w
  \rrangle}&=-\llangle \D^{2}_{+}\phi^{(k)}(t,0)+
            \D^{2}_{+}\psi^{(k)}(t,0)^{*}(x), v\otimes w\rrangle \nonumber\\
  &\quad+\llangle \D_{+}\phi(t,0)\otimes \D_{+}\phi(t,0),v\otimes w\rrangle\nonumber\\
  &\quad+\llangle \D_{+}\phi(t,0)\otimes
    \D_{+}\psi(t,0)^{*}(x),v\otimes w\rrangle\nonumber\\
  &\quad+\llangle \D_{+}\psi(t,0)(x)\otimes \D_{+}\phi(t,0) ,
    v\otimes w\rrangle\nonumber\\
  &\quad+\llangle \D_{+}\psi(t,0)^{*}(x)\otimes
    \D_{+}\psi(t,0)^{*}(x), v\otimes w\rrangle,  
\end{align}
where we conveniently identified functionals on $\cH$ with
elements of $\cH$.  Written in this form the right-hand side in formula
\eqref{eq:explicit-Xv-square-k} reveals its linearity in $v\otimes w$ and for
$v\otimes w\in \cL_{2}(\cH)$, we have
\begin{align*}
v\otimes w=v^{+}\otimes w^{+}-v^{+}\otimes w^{-}-v^{-}\otimes w^{+}+v^{-}\otimes w^{-}  
\end{align*}
and thus expanding both sides by
linearity in equation \eqref{eq:explicit-Xv-square-k}, shows the validity of
the formula for all $v,w\in\cH$. 
\end{proof}
Note that by inserting the formulas from
\eqref{eq:variational-solution-3-0}--\eqref{eq:variational-solution-3-1} and~\eqref{eq:def_dR}--\eqref{eq:def_d2F} into
the corresponding terms in \eqref{eq:explicit-Xv-k} and
\eqref{eq:explicit-Xv-square-k}, the latter become explicit up to the parameters
$(b,B,m,\mu)$. To save some space, we give those explicit formulas only for the limit case in Proposition \ref{prop:explicit-formula} below.\\ 
Using the formulas from Lemma~\ref{lem:uniform-bound-k}, we establish
uniform growth bounds for the semigroups $(P^{(k)}_t)_{t\geq 0}$ in the next
proposition.
Let us note here that in general we do not obtain an uniform growth bound $w\in\MRplus$ with $M=1$:
\begin{proposition}\label{prop:uniform-bound-k} 
Let $(b,B,m,\mu)$ be an admissible parameter set conform
Definition~\ref{def:admissibility} and for every $k\in\MN$ let
$(P^{(k)}_t)_{t\geq 0}$ be the generalized Feller semigroup on
$\cB_{\rho}(\cH^+_{\textnormal{w}})$ associated with $(b,B,m^{(k)},\mu^{(k)})$, the
existence of which is guaranteed by Proposition~\ref{prop:Gk_generator}. Then
there exists a constant $w \in \MRplus$ and $M\geq 1$, both independent of $k\in\MN$, such that
\begin{align}\label{eq:quasi-contractive-uniform}
 \| P_t^{(k)} \|_{\cL(\cB_{\rho}(\cH^+_{\textnormal{w}}))} \leq M \e^{w t} \quad\text{for all}\; k\in \N, t\geq 0.
\end{align}
\end{proposition}
\begin{proof}
  Recall from Remark~\ref{rem:semigroup-on-b-rho}, that in order to show the
  existence of a $M\geq 1$ and $w\in\MRplus$ such that
  equation~\eqref{eq:quasi-contractive-uniform} holds, it suffices to show the
  existence of a $\epsilon>0$ and  $C\geq 0$, independent of $k\in\MN$, such that
\begin{align}\label{eq:quasi-contractive-uniform-suffice}
 \EXspec{x}{\rho(X_{t}^{(k)})}\leq C \rho(x)\,,\quad\forall t \in [0,\epsilon]\text{ and
  }x\in\cHplus.  
\end{align}
Let $k\in\MN$ be arbitrary, but fixed and denote by $(e_{n})_{n\in\MN}$ an ONB of
$\cH$, then by Parseval's identity and monotone convergence we have:
\begin{align*}
\EXspec{x}{\rho(X_{t}^{(k)})}=\EXspec{x}{1+\norm{X_{t}^{(k)}}^{2}}=1+\sum_{n=1}^{\infty}\EXspec{x}{\langle
  X_{t}^{(k)},e_{n}\rangle^{2}}\,,
\end{align*}
for every $t\geq 0$ and $x\in\cHplus$. By
equation~\eqref{eq:explicit-Xv-square-k}, in particular using the notation in
equation~\eqref{eq:explicit-Xv-square-k-tensor}, we have for all $n\in\MN$: 
\begin{align}
\EXspec{x}{\langle X_{t}^{(k)},e_{n}\rangle^{2}}&=\llangle
                                       -\D^{2}_{+}\phi^{(k)}(t,0)-\D^{2}_{+}\psi^{(k)}(t,0)^{*}(x),e_{n}\otimes e_{n}
                                       \rrangle \nonumber\\
  &\qquad+\llangle
    \big(\D_{+}\phi(t,0)+\D_{+}\psi(t,0)^{*}(x)\big)^{\otimes 2},e_{n}\otimes e_{n}\rrangle.\label{eq:uniform-bound-k-2} 
\end{align}
We show separately for the first and second terms on the right-hand side of equation~\eqref{eq:uniform-bound-k-2} 
that, when
summing over all $n\in\MN$, we find a $\epsilon>0$ and $C\geq 0$ such that equation~\eqref{eq:quasi-contractive-uniform-suffice} holds. Since
\begin{align*}
 \sum_{n=1}^{\infty}\langle
                                  \D_{+}\phi(t,0)+\D_{+}\psi(t,0)^{*}(x),e_{n}\rangle^{2}=\norm{
                                  \D_{+}\phi(t,0)+\D_{+}\psi(t,0)^{*}(x)}^{2}, 
\end{align*}
we deduce for the second term ion the right hand side of~\eqref{eq:uniform-bound-k-2}:
\begin{align*}
&\sum_{n=1}^{\infty}\llangle
  \big(\D_{+}\phi(t,0)+\D_{+}\psi(t,0)^{*}(x)\big)^{\otimes
                 2},e_{n}\otimes e_{n}\rrangle\leq  C(t)(1+\norm{x}^{2}),                               
\end{align*}
for 
\begin{align*}
C(t)=\big(\norm{\D_{+}\phi(t,0)}+\norm{\D_{+}\psi(t,0)^{*}}_{\cL(\cH)}\big)^{2}.  
\end{align*}
The terms $\norm{\D_{+}\phi(t,0)}$ and $\norm{\D_{+}\psi(t,0)^{*}}_{\cL(\cH)}$
are bounded for all $t\geq 0$.
Therefore, we deduce the existence of $\epsilon>0$ and
$C\geq 0$, independent of $k\in\MN$, such that  
\begin{align}\label{eq:uniform-bound-k-1-latter}
\sum_{n=1}^{\infty}\llangle
  \big(\D_{+}\phi(t,0)+\D_{+}\psi(t,0)^{*}(x)\big)^{\otimes
  2},e_{n}\otimes e_{n}\rrangle\leq C (1+\norm{x}^{2}),
\end{align}
for all $t\in[0,\epsilon]$ and $x\in\cHplus$.
We continue with the first term on the right hand side of~\eqref{eq:uniform-bound-k-2}. Recall 
formulas~\eqref{eq:def_dF},~\eqref{eq:def_d2F},~\eqref{eq:dif-eq-phi-2},~\eqref{eq:variational-solution-3-0} and~\eqref{eq:variational-solution-3-1}, from which we obtain:
\begin{align}\label{eq:uniform-bound-k-3}
  &\llangle \D^{2}_{+}\psi^{(k)}(t,0)^{*}(x), e_{n}\otimes e_{n} \rrangle \nonumber\\
  &\quad= -\int_{0}^{t}
        \int_{\cHpluso}\langle\E^{s \D
                                                                           R(0)^{*}}\xi,e_{n}\rangle^{2}\big\langle
                                                                           x,\E^{(t-s)\D
                                                                           R(0)} \big\rangle\frac{\dmuk}{\norm{\xi}^{2}}
                                                                          \D
                                                                           s,
\end{align}
and 
\begin{align}\label{eq:uniform-bound-k-4}
\llangle \D^{2}_{+}\phi^{(k)}(t,0) , e_{n}\otimes e_{n} \rrangle &=-\int_{0}^{t} \big(\int_{\cHpluso}
        \langle 
                                                             \E^{s \D R(0)^{*}}\xi, e_{n}\rangle^{2}\, \dmk \nonumber\\
  &\quad+ 
   \llangle \D^{2}_{+}\psi^{(k)}(s,0)^{*}(b), e_{n}\otimes e_{n}\rrangle 
\,\big)\D s
  \nonumber\\
  &\quad 
  + 
\int_{0}^{t} 
    \int\limits_{\cHplus\cap \{\norm{\xi}\geq 1\}}
   \langle \langle
    \D^{2}_{+}\psi^{(k)}(s,0)^{*}(\xi), e_{n}\otimes e_{n}\rangle \rangle
    \,\dm \D s\,.
\end{align}
Hence the two terms on the right hand side of equation~\eqref{eq:uniform-bound-k-2} can be estimated by
\begin{align*}
  \sum_{n=1}^{\infty}&\int_{0}^{t}\int_{\cHpluso}\langle\E^{s \D R(0)^{*}}\xi,e_{n}\rangle^{2}\langle
  x,\E^{(t-s)\D
  R(0)}\rangle\frac{\dmuk}{\norm{\xi}^{2}}
  \D
  s\\
  &\leq\big(\int_{0}^{t}\norm{\E^{s
  \D R(0)^{*}}}_{\cL(\cH)}^{2}\norm{\E^{(t-s)\D
  R(0)}}_{\cL(\cH)}\norm{\mu(\cHpluso)}\D s\big)\norm{x}, 
\end{align*}
and 
\begin{align*}
 &\sum_{n=1}^{\infty}\llangle \D^{2}_{+}\phi^{(k)}(t,0) , e_{n}\otimes e_{n}
   \rrangle\\
&\quad \leq 2(\norm{b}+\norm{\mu(\cHpluso)}+\int_{\cHpluso}\norm{\xi}^{2}+\norm{\xi-\chi(\xi)}\dm)\\
  &\quad\quad\times \int_{0}^{t}\int_{0}^{s}\norm{\E^{\tau \D R(0)^{*}}}^{2}_{\cL(\cH)}\norm{\E^{(s-\tau)\D R(0)}}_{\cL(\cH)} \D \tau\D s\,,
\end{align*}
where we used that for all $k\in\MN$:
\begin{align*}
\norm{\mu^{(k)}(\cHpluso)}\leq \norm{\mu(\cHpluso)}<\infty  
\end{align*}
and
\begin{align*}
\int_{\cHpluso}\norm{\xi}^{2}+\norm{\xi-\chi(\xi)}\dmk\leq \int_{\cHpluso}\norm{\xi}^{2}+\norm{\xi-\chi(\xi)}\dm<\infty.  
\end{align*}
Therefore there exist $\epsilon>0$ and $\tilde{C}\geq 0$ such that
\begin{align*}
&\sum_{n=1}^{\infty}\llangle
  -\D^{2}_{+}\phi^{(k)}(t,0)-\D^{2}_{+}\psi^{(k)}(t,0)^{*}(x),e_{n}\otimes
  e_{n}\rrangle\\
  &\quad\quad\leq \tilde{C} (1+\norm{x}^{2}),
\end{align*}
for all $t \in [0,\epsilon]$ and $x\in\cHplus$. Taking the sum of the latter constant $\tilde{C}$ and the constant $C$ found in
equation~\eqref{eq:uniform-bound-k-1-latter} yields~\eqref{eq:quasi-contractive-uniform-suffice}.
\end{proof}
In the next step we show that the family $(P_{t})_{t\geq 0}$, defined by
$P_t\coloneqq\lim_{k\to\infty}P^{(k)}_t$ for $t\geq 0$, gives rise to a generalized Feller semigroup  
and deduce the existence of a generalized Feller process $(X_t)_{t\geq 0}$
with generator $\mathcal{G}$ as in formula~\eqref{eq:affine-generator-form}.
 \begin{proposition}\label{prop:affinejump-genFeller}
Let $(b,B,m,\mu)$ be an admissible parameter set conform Definition~\ref{def:admissibility}. Then there exists a generalized Feller
   semigroup $(P_t)_{t\geq 0}$ on $\cB_{\rho}(\cH^+_{\textnormal{w}})$ such that 
   \begin{align}\label{eq:expo-affine-form}
     \left( P_t\E^{-\langle \cdot,u\rangle}\right)(x)=\E^{-\phi(t,u)-\langle x,\psi(t,u)\rangle},
   \end{align}
   for all $t\geq 0$ and $x, u\in\cHplus$, where
   $(\phi(\cdot,u),\psi(\cdot,u))$ is the unique solution to
   the generalized Riccati equation~\eqref{eq:riccati-conc}. The semigroup $(P_t)_{t\geq 0}$ gives rise to
   a generalized Feller process $(X_{t})_{t\geq 0}$ in
   $(\cHplus_{\textnormal{w}},\norm{\cdot}^{2}+1)$ such that
   \begin{align*}
     \EXspec{x}{f(X_{t})}=P_tf(x),  \quad t\geq 0, \quad x\in \mathcal{H}^+\,,
   \end{align*}
   and the generator $\cG$ of $(P_t)_{t\geq 0}$ is of the form in
   equation~\eqref{eq:affine-generator-form} on $\cD$. 
 \end{proposition}
 \begin{proof}
 Hereto we check that the conditions of Theorem 3.2 in \cite{CT20} hold.
 From
Proposition~\ref{prop:uniform-bound-k}, we know that the sequence of semigroups $(P^{(k)}_t)_{t\geq 0, k\in\MN}$
with generators $(\cG^{(k)})_{k\in\MN}$ satisfy the following growth bound 
\begin{align}\label{eq:uni-growth-bound}
  \norm{P^{(k)}_t}_{\cL(\cB_{\rho}(\cH^+_{\textnormal{w}}))}\leq M \E^{w t}, \qquad \forall \,\, n \in \mathbb{N} \quad \mbox{and}\,\, t\geq 0\,,
\end{align}
where $w\in\MR$.

Recall the definition of $\mathcal{D}$ from equation~\eqref{eq:span-Fourier}
and recall from Lemma~\ref{lem:dense_subset_Brho} that $\mathcal{D}$ is a
dense subspace of $\cB_{\rho}(\cH^+_{\textnormal{w}})$. Thus (i) in Theorem
3.2 in \cite{CT20} is satisfied.

Note that the operator $\cG^{(n)}$, $n\in \mathbb{N}$, applied to the function
$\E^{-\phi^{(k)}(s,u)-\langle \cdot,\psi^{(k)}(s,u)\rangle}$, with
$(\phi^{(k)}(\cdot,u), \psi^{(k)}(\cdot,u))$ being a solution to
\eqref{eq:riccati-conc-k}, gives (see also equation~\eqref{eq:Kolomgorov-Abs-k})
\begin{align*}
&\left(\cG^{(n)}\E^{-\phi^{(k)}(s,u)-\langle \cdot,\psi^{(k)}(s,u)\rangle}\right)(x)\\
&\qquad = \E^{-\phi^{(k)}(s,u)}\cG^{(n)}\E^{-\langle \cdot,\psi^{(k)}(s,u)\rangle}(x)
  \\
  &\qquad  =\left(-F^{(n)}(\psi^{(k)}(s,u))-\langle
                                                                      R^{(n)}(\psi^{(k)}(s,u)),x\rangle\right)\E^{-\phi^{(k)}(s,u)-\langle x,\psi^{(k)}(s,u)\rangle}\,,
\end{align*}
for $x,u \in \mathcal{H}^+$, $s\geq 0$
From the latter and equation~\eqref{eq:affine-X-tilde-k}, we infer
\begin{align}\label{eq:upper-bound-a-b}
& \frac{1}{\|x\|^2+1}\left\lvert \cG^{(n)}P^{(k)}_{s}\E^{-\langle \cdot,u\rangle}(x)-\cG^{(k)}P^{(k)}_{s}\E^{-\langle
 \cdot, u\rangle}(x)\right\rvert \nonumber \\ 
  &\leq \frac{\E^{-\phi^{(k)}(s,u)-\langle
                      x,\psi^{(k)}(s,u)\rangle}}{\|x\|^2+1} \left[b_{s,u}^{(n,k)} +\|x\| a_{s,u}^{(n,k)}\right]\,,
\end{align}
where
\begin{align*}
  a_{s,u}^{(n,k)}&:= \left\|R^{(n)}(\psi^{(k)}(s,u))-R^{(k)}(\psi^{(k)}(s,u))\right\|
  \end{align*}
  and 
  \begin{align*}
  b_{s,u}^{(n,k)}&:=\left \lvert F^{(n)}(\psi^{(k)}(s,u))-F^{(k)}(\psi^{(k)}(s,u))\right \rvert 
  \end{align*}
From the equations~\eqref{eq:exp_est} and~\eqref{eq:upper-bound-phi} we have,
for all $0\leq s\leq T<\infty$: 
\begin{align*}
  & \left\lvert \left(\E^{-\langle \xi, \psi^{(k)}(s,u)\rangle}-1-\langle  \chi(\xi),\psi^{(k)}(s,u)\rangle\right)\left(\one_{\{\norm{\xi}>1/n\}}-\one_{\{\norm{\xi}>1/k\}}\right)\right\rvert \\
   &\quad \leq \|\psi^{(k)}(s,u)\|^2  \|\xi\|^2 \one_{\{{\norm{\xi}\leq 1}\}}\\
     &\quad \leq \sup_{s\in [0,T]}\|\psi^{(1)}(s,u)\|^2 \|\xi\|^2  \one_{\{{\norm{\xi}\leq 1}\}}=: g(\xi)\,.                                                                                 
\end{align*}
Observe that for $h \in \mathcal{H}^+$, we have $\int_{\cHpluso} g(\xi)\, \frac{\langle \dmu, h\rangle}{\|\xi\|^2} <\infty$\,.
Hence Lemma~\ref{lem:characint} implies that $g(\cdot)/\|\cdot\|^2 \in \mathcal{L}^1(\cHplus, \mu)$ and from Theorem~\ref{thm:DCT}, we deduce that 
$\sup_{s\in [0,T]}a^{(n,k)}_{s,u}$ converges to $0$ as $n,k\to\infty$.
By the admissibility condition~\ref{eq:m-2moment} in Definition~\ref{def:admissibility}, we infer $\int_{\cHpluso} g(\xi)\,\dm <\infty$
and applying the dominated convergence theorem we also deduce that
$\sup_{s\in [0,T]}b_{s,u}^{(n,k)}$ converge to $0$ as $n,k\to\infty$.
Observing that $\phi^{(k)}(s,u)\in\MRplus$ and $\psi^{(k)}(s,u) \in
\mathcal{H}^+$ for all $s\geq 0$, we can bound $\E^{-\phi^{(k)}(s,u)-\langle
                      x,\psi^{(k)}(s,u)\rangle}$ by $1$ for all $x\in\cHplus$
                    and get from equation~\eqref{eq:upper-bound-a-b}, that for all $s>0$:
\begin{align*}
&\left\norm{\cG^{(n)}P^{(k)}_{s}\E^{-\langle \cdot,u\rangle}-\cG^{(k)}P^{(k)}_{s}\E^{-\langle
  \cdot, u \rangle}\right}_{\rho}\nonumber \\ 
  &\qquad \leq
                        \sup_{x\in\cHplus}\frac{\norm{x}+1}{\norm{x}^{2}+1}\left(a_{s,u}^{(n,k)}+b_{s,u}^{(n,k)}\right)\nonumber \\
                      &\qquad \leq \left(\sup_{s\in [0,T]}a_{s,u}^{(n,k)}+\sup_{s\in [0,T]}b^{(n,k)}_{s,u}\right)C_u
                        \norm{\E^{-\langle \cdot, u \rangle}}_{\infty}\,,  
\end{align*}
where $C_u = \sup_{x\in\cHplus}(\norm{x}+1)/(\norm{x}^{2}+1)$. Thus
condition (ii) in Theorem 3.2 in \cite{CT20} is satisfied with
$\norm{\cdot}_{\cD}=\norm{\cdot}_{\infty}$
and we deduce the existence of a generalized Feller semigroup
$(P_t)_{t\geq 0}$ with the same growth bound as the semigroup
$(P^{(k)}_t)_{t\geq 0}$ and such that $P_tf=\lim_{k\to\infty}P^{(k)}_tf$, for
all $f\in\cB_{\rho}(\cH^+_{\textnormal{w}})$, uniformly on compacts in time. Since $P_t1=1$, for all
$t\geq 0$, we  deduce from Theorem~\ref{thm:Kol-extension-theorem} that there 
exists a generalized Feller process $(X_{t})_{t\geq 0}$ such that
$P_tf(x)=\EXspec{x}{f(X_{t})}$  for all $t\geq 0$ and $x\in \mathcal{H}^+$.
The exponential affine formula~\eqref{eq:expo-affine-form} follows from
formula~\eqref{eq:affine-X-tilde-k} and the fact that
$\lim_{k\to\infty}\phi^{(k)}(t,u)=\phi(t,u)$ and
$\lim_{k\to\infty}\psi^{(k)}(t,u)=\psi(t,u)$ for all $t\geq 0$ and
$u\in\cHplus$. From this we further derive the particular form of the
generator $\cG$ on the space $\cD$ by noting that $t\mapsto P_{t}\E^{-\langle
  \cdot,u\rangle}(x)$ uniquely solves the abstract Cauchy problem associated
to $(\cG,\dom(\cG))$ and hence by mimicking the proof of the approximation case in
Proposition~\ref{prop:Gk_generator}, we conclude formula \eqref{eq:affine-generator-form}. 

\end{proof}

Analogous to the approximating processes $(X^{(k)}_{t})_{t\geq 0}$,
for $k\in\MN$ in Lemma~\ref{lem:uniform-bound-k}, we now deduce explicit formulas for the
expressions $\EXspec{x}{\langle X_{t}, v\rangle}$ as well as for
$\EXspec{x}{\langle X_{t},v \rangle^{2}}$, where $x\in\cHplus$, $t\geq 0$ and
$v\in\cHplus$. 
\begin{proposition}\label{prop:explicit-formula}
Let $(b,B,m,\mu)$ be an admissible parameter set conform Definition~\ref{def:admissibility}.
  Recall the definition of $\D R(0), \D^2 R(0)$, $\D F(0)$, and $\D^2 F(0)$ from~\eqref{eq:def_dR}--\eqref{eq:def_d2F}.
 Then  for all $v,w\in\cHplus$ the following formulas hold true: 
\begin{align}\label{eq:explicit-Xv}
\EXspec{x}{\langle X_{t}, v\rangle}&=\int_{0}^{t}\langle b, \E^{s\D R(0)}v\rangle +
          \int_{\cHplus\cap \{\norm{\xi}> 1\}}\langle
    \xi,\E^{s \D R(0)}v\rangle\,
    m(\D\xi)\D s+\langle x, \E^{t\D R(0)}v\rangle  
\end{align}
and 
\begin{align}\label{eq:explicit-Xv-square}
  & \EXspec{x}{\langle X_{t}, v\rangle \langle X_{t}, w\rangle }
\nonumber\\ & \quad 
   = -\int_0^t \D^2 F(0)(\E^{s \D R(0)} v, \E^{s \D R(0)} w ) \,\D s
\nonumber \\ & \quad \quad 
    - \int_0^t \int_{0}^{s} 
        \D F(0) \left( 
            \E^{(s-u)\D R(0) }
            \D^2 R(0) (
                \E^{u \D R(0)} v, \E^{u \D R(0)} w
                )
            \right)
      \,\D u\, \D s
\nonumber\\ & \quad \quad 
              - \int_{0}^{t}
        \left\langle x,
            \E^{(t-s)\D R(0) }
            \D^2 R(0)(
                \E^{s \D R(0)} v, \E^{s \D R(0)} w
            )
        \right\rangle
        \,\D s
\nonumber\\ & \quad \quad 
    + \left( 
        \int_0^t \D F(0) (\E^{s \D R(0)} v) \,\D s 
        + 
        \left\langle x,
        \E^{t \D R(0)} v
        \right\rangle
    \right)
\nonumber\\ & \quad \quad \quad  
    \times \left(
        \int_0^t \D F(0) (\E^{s \D R(0)} w) \,\D s 
        + 
        \left\langle x,
        \E^{t \D R(0)} w
        \right\rangle
    \right). 
\end{align}
Moreover, for $v \in \cHplus$, $\langle \cdot, v\rangle \in \dom(\cG)$ and 
\begin{equation}\label{eq:generator-of-xv}
\cG\langle \cdot, v\rangle(x) = \langle b+B(x), v\rangle +\int_{\cHplus \cap \{\|\xi\|>1\}}\langle \xi, v\rangle\, \nu(x, \D\xi)\,, \qquad x \in \cHplus\,.
\end{equation}
\end{proposition}
\begin{proof}
Formulas~\eqref{eq:explicit-Xv} and~\eqref{eq:explicit-Xv-square} can be obtained analogous to the computation of the formulas \eqref{eq:explicit-Xv-k} and \eqref{eq:explicit-Xv-square-k} derived for the approximating case, combined with the explicit formulas~\eqref{eq:variational-solution-3-0}--\eqref{eq:variational-solution-3-1}.
As in the proof of Lemma~\ref{lem:uniform-bound-k} we use Proposition
\ref{prop:gateaux-diff} and the finite second moments of the process $(X_{t})_{t\geq 0}$ to
interchange the operations of the expectation and the one-sided
derivatives. 
To obtain more explicit formulas, we consider the analogous of the formulas \eqref{eq:explicit-Xv-k} and \eqref{eq:explicit-Xv-square-k} and recall that $\D_{+}\phi(t,0)(v)$, $\D^{2}_{+}\phi(t,0)(v,v)$ can be expressed in terms of $\D F(0)$, $\D^2 F(0)$, $\D_{+}\psi(t,0)(v)$, and $\D_{+}^2\psi(t,0)(v,w)$, see~\eqref{eq:dif-eq-phi-1} and~\eqref{eq:dif-eq-phi-2}. Then, we recall the expressions \eqref{eq:variational-solution-3-0} and \eqref{eq:variational-solution-3-1} for $\D_{+}\psi(t,0)(v)$, and $\D_{+}^2\psi(t,0)(v,w)$.

To prove \eqref{eq:generator-of-xv}, observe that using the analogue of \eqref{eq:explicit-Xv-k}, we get
\begin{align*}
&\frac{1}{t}\Big|P_t \langle \cdot, v\rangle(x) -\langle x, v\rangle - \langle b+B(x), v\rangle -\int_{\cHplus \cap \{\|\xi\|>1\}}\langle \xi, v\rangle\, \nu(x, \D\xi)\Big|\\
&\qquad \leq \frac{1}{t}\Big|\D_+\phi(t,0)(v) - \langle b, v\rangle -\int_{\cHplus \cap \{\|\xi\|>1\}}\langle \xi, v\rangle\, m(\D \xi)\Big|\\
&\qquad \qquad +\frac{1}{t} \|x\|\Big\|\D_+\psi(t,0)(v) - v- B^*(v)- \int_{\cHplus \cap \{\|\xi\|>1\}} \langle\xi, v\rangle\, \frac{\mu(\D \xi)}{\|\xi\|^2}\Big\|\,.
\end{align*}
The latter together with formulas \eqref{eq:dif-eq-phi-1} and \eqref{eq:dif-eq-psi-1}, yield 
\begin{align*}
&\lim_{t\rightarrow 0^+}\sup_{x \in \cHplus} \frac{\frac{1}{t}\Big|P_t \langle \cdot, v\rangle(x) -\langle x, v\rangle - \langle b+B(x), v\rangle -\int_{\cHplus \cap \{\|\xi\|>1\}}\langle \xi, v\rangle\, \nu(x, \D\xi)\Big|}{1+\|x\|^2}\\
&\qquad \leq  \Big|\D F(0)(v)  - \langle b, v\rangle -\int_{\cHplus \cap \{\|\xi\|>1\}}\langle \xi, v\rangle\, m(\D \xi)\Big| \\
&\qquad \qquad + \Big\|\D R(0)(v) - B^*(v)- \int_{\cHplus \cap \{\|\xi\|>1\}} \langle\xi, v\rangle\, \frac{\mu(\D \xi)}{{\|\xi\|^2}}\Big\|
\end{align*}
and recalling the formulas for $\D R(0)$ and $\D F(0)$ respectively in \eqref{eq:def_dR} and \eqref{eq:def_dF}, we conclude that $\langle \cdot, v\rangle \in \dom(\cG)$, for $v \in \cHplus$ and that \eqref{eq:generator-of-xv} holds.
\end{proof}
\begin{remark}\label{rem:second-moment}
\hfill
As observed in Remark \ref{rem:comparison_finite_dim_admissability}, the second moment conditions are a consequence of our generalized Feller approach with weight function $\rho=\| \cdot \|^2 +1$. More specifically, the uniform bounds established in Proposition~\ref{prop:uniform-bound-k} rely on the existence of second moments as established in Lemma~\ref{lem:uniform-bound-k}. A natural question to ask is whether one could perform the analysis with a different (weaker) weight function. However, in the proof of Lemma~\ref{lem:X-semigroup} we consider the square root of the weight function $\sqrt{\rho}$,
more specifically, we need that $\sqrt{\rho(x)} \geq c\| x \|$, $x\in \cH^+$, for some constant $c\in (0,\infty)$.\\     
Naturally, the second moments of $m$ and $\mu$ are also used to derive the explicit formulas for the first and second moments of the affine process in Proposition~\ref{prop:explicit-formula}. Finally, we note that the existence 
of a \emph{first} moment of $\frac{\dmu}{\| \xi \|^2}$ is already used in Lemma~\ref{lemma:mono-Lipschitz} to ensure that the approximating mappings $R^{(k)}$ are Lipschitz continuous. 
\end{remark}

In general we do not obtain a version of the process $X$ in Proposition~\ref{prop:affinejump-genFeller} with c\`adl\`ag paths. By Theorem 2.13 in \cite{CT20} a c\`adl\`ag version exists when the associated semigroup $(P_{t})_{t\geq 0}$ is quasi-contractive on $\cB_{\rho}(\cH^+_{\textnormal{w}})$, i.e., if one can take $M=1$ in Proposition \ref{prop:uniform-bound-k}. We do not know whether this holds in general. However, we can show that $X$ admits a c\`adl\`ag version in the finite activity setting:

\begin{proposition}\label{prop:cadlag}
Assume the setting of Proposition~\ref{prop:affinejump-genFeller} 
and assume moreover that $m(\cHpluso)<\infty$ and that $\cH^{+}\setminus \{0\} \ni \xi \mapsto \| \xi \|^{-2}$ is $\mu$-integrable. Then there exists a version of $X$ with c\`adl\`ag paths.
\end{proposition}

\begin{proof}
By~\cite[Theorem 2.13]{CT20} it in fact suffices to prove that the generalized
Feller semigroup $(P_t)_{t\geq 0}$ associated to $X$ is quasi-contractive on
$\cB_{\tilde{\rho}}(\cH^+_{\textnormal{w}})$, where $\tilde{\rho}\colon \cH^+
\rightarrow [0,\infty)$ is an admissible weight function such that its
associated norm $\norm{\cdot}_{\tilde{\rho}}$ is equivalent to $\norm{\cdot}_{\rho}$. Note that in the finite activity setting we can apply Proposition~\ref{prop:Gk_generator} with $k=\infty$ (with the understanding that $m^{(\infty)} := m$ and $\mu^{(\infty)}:= \mu$) to directly obtain $(P_t)_{t\geq 0}$ (i.e., no approximation over $k$ is necessary). In particular $\tilde{\omega}_{\infty}<\infty$, where $\tilde{\omega}_{\infty}$ is  
defined by taking $k=\infty$ in~\eqref{eq:tilde_omega_est}. It then follows from statement a) on page~\pageref{it:quasi-contractive} that $(P_t)_{t \geq 0}$ is quasi-contractive 
on $\cB_{\tilde{\rho}_{\infty}}(\cH^+_{\textnormal{w}})$ where
$\tilde{\rho}_{\infty}$ is an admissible weight function with associated norm
equivalent to $\norm{\cdot}_{\rho}$.
\end{proof}

  In the next section we give the proof of Theorem~\ref{thm:existence-affine-process}. The proof is based on collecting the results
  from this section and transferring from a generalized Feller setting to the classical
  setting that we used for presenting the results in Section~\ref{sec:setting-main-result}.
\subsection{Proof of Theorem \ref{thm:existence-affine-process}}\label{proof:existence-affine-process}
    Let $(b,B,m,\mu)$ be an admissible parameter set. Then by
    Proposition~\ref{prop:affinejump-genFeller} there exists a generalized Feller semigroup
    $(P_{t})_{t\geq 0}$ and the associated generalized Feller process
    $(X_{t})_{t\geq 0}$ in $\cHplus$ such that
    \begin{align*}
     \EXspec{x}{f(X_{t})}=P_{t}f(x)\quad \text{for}\;t\geq 0,
    \end{align*}
    and the Markov property \eqref{eq:Kol-extension-theorem-1} holds. The
    existence of constants $M,\omega \in [1,\infty)$ such
    that~\eqref{eq:exp_bound} is satisfied follows from
    Remark~\ref{rem:semigroup-on-b-rho}. The space $\cH$ is a separable Hilbert
    space and hence the Borel-$\sigma$-algebras $\cB(\cHplus)$ and
    $\cB(\cHplus_{\textnormal{w}})$ coincide. This means that the transition kernels
    $(p_{t}(x,\D y))_{t\geq 0}$ defining the semigroup $(P_{t})_{t\geq 0}$
    stay unaffected under the change of topology and hence the process
    $(X_{t})_{t\geq 0}$ is also a Markov process in $\cHplus$ with the strong
    topology.\\
    The asserted exponential-affine formula in \eqref{eq:laplace-transform}
    is precisely formula~\eqref{eq:expo-affine-form} from Proposition
    \ref{prop:affinejump-genFeller}. By this and
    Proposition~\ref{prop:exist-uniq-solut} we have for all $x\in\cHplus$:     
    \begin{align}\label{eq:weak-generator-last}
     \lim_{t\to 0+}\frac{P_{t}\E^{-\langle \cdot,u\rangle}(x)-\E^{-\langle
      \cdot, u\rangle}(x)}{t}&=\lim_{t\to 0+}\frac{\E^{-\phi(t,u)-\langle
      x,\psi(t,u)\rangle}-\E^{-\langle x,u\rangle}}{t}\nonumber\\
      &=(-F(u)-\langle x,
      R(u)\rangle)\E^{-\langle x, u\rangle}. 
    \end{align}
    In particular, we see that $\cA(\cD)\subseteq C_{b}(\cHplus)$ and since
    $(P_{t})_{t\geq 0}$ is a strongly continuous semigroup on
    $\cB_{\rho}(\cHplus_{\textnormal{w}})$ we have $\left( P_{t}\E^{-\langle \cdot,u\rangle}\right)(x)=\E^{-\langle \cdot,
      x\rangle}(x)+\int_{0}^{t}
      \left( P_{s}\cA \E^{-\langle \cdot, u\rangle}\right)(x)\D s$. 
    Consequently, we have shown that $\cD\subseteq \dom(\cA)$ and from
    formula~\eqref{eq:weak-generator-last} we see that
    formula~\eqref{eq:affine-generator-form} holds true on $\cD$.
\section{Conclusions and Outlook}\label{sec:conclusions-outlook}
With Theorem~\ref{thm:existence-affine-process} we have proven the existence of affine Markov processes in
the cone of positive self-adjoint Hilbert-Schmidt operators by a
novel approach inspired by~\cite{CT20}. 
In particular, our approach relies on the theory of generalized Feller processes, taking the weight function $\rho=\norm{\cdot}^{2}+1$. This approach requires the existence of first and second moments of the jump measures $m$ and $\mu$. A beneficial by-product is that we obtain explicit formulas for the first and second moments of the affine Markov process, see Proposition~\ref{prop:explicit-formula}. See Remark~\ref{rem:second-moment} for a discussion regarding the necessity of the second-moment condition.\par
Below, we discuss and motivate three further directions of research.
\subsubsection*{{\emph On relaxing the condition on existence of moments.}} A possible direction of further research is to investigate whether one can adapt the proof in such a way to allow for the weight function
$\rho=\norm{\cdot}+1$. In this case a first moment conditions on $m$ and $\mu$
should suffice. On a more abstract level, the question arises whether it is
possible to establish existence without any moment conditions, as can be done
in the finite dimensional setting where the cone of interest does not have
empty interior. Another tantalizing question is to what degree an infinite dimensional affine process on the cone of positive self-adjoint Hilbert-Schmidt operators allows for diffusion. It is clear from~\cite{BS18} that certain constructions are possible. 
\par 
\subsubsection*{{\emph On the construction of stochastic volatility models.}} Our main motivation for considering affine processes on the space of positive
self-adjoint Hilbert-Schmidt operators is that such processes qualify as
infinite dimensional stochastic covariance processes. Hence we consider in
\cite{cox2021infinitedimensional} stochastic volatility models in Hilbert spaces, where the introduced class of affine pure-jump
processes will be used for modeling the operator-valued instantaneous
variance process.
Specifically, we will consider a process $(Y_{t})_{t\geq 0}$ in a Hilbert space $(H, \langle \cdot, \cdot\rangle)$ given by 
\begin{align}\label{eq:stochastic-vola-model}
 \D Y_{t}=\cA Y_{t}\, \D t+\sigma_{t}Q^{1/2}\,\D W_{t}, \qquad t\geq0\,,
\end{align}
where $\cA\colon\dom(\cA)\subseteq H\to H$ is a possibly unbounded operator
with dense domain $\dom(\cA)$, $(W_{t})_{t\geq 0}$ is a cylindrical Brownian
motion in $H$, $Q\in \mathcal{H}^+$, and $(\sigma_{t})_{t\geq 0}$
is an operator valued stochastic process given by the square-root of an affine 
pure-jump process, the existence of which is guaranteed by our main result Theorem \ref{thm:existence-affine-process}.\par 
\subsubsection*{{\emph On considering a different state space for the covariance process.}} Note that we take $\sigma$ in~\eqref{eq:stochastic-vola-model} to be the \emph{square root} of an affine process in order to obtain that $Y$ is again affine. However, this means that the `natural' state space for $\sigma$ is not the cone of positive self-adjoint Hilbert-Schmidt operators, but the cone of positive self-adjoint \emph{trace class} operators. Unfortunately, this is no longer a cone in a Hilbert space. As self-duality of the cone was used at various instances in the proof of Theorem~\ref{thm:existence-affine-process}, it is not clear how much can be salvaged if we consider trace class operators. This would be a further interesting direction of research.
\appendix
\section{A comparison theorem}\label{sec:comparison-theorem}
A more general version of the following comparison theorem can be found, e.g., as~\cite[Theorem 5.4]{Dei77}. 
\begin{theorem}\label{thm:comparison-theorem}
Let $(H, (\cdot, \cdot))$ be a Hilbert space, $K \subset H$ a cone, let $T>0$, and let $F\colon [0,T]\times H\rightarrow H$. Assume that $F(t,\cdot)$ is quasi-monotone with respect to $K$
for all $t\in [0,T]$, and that there exists a constant $L\in [0,\infty)$ such that
\begin{equation}
 \| F(t,x) - F(t,y)\|_{H} \leq L \| x-y\|_H,\quad t\in [0,T], x,y\in H.
\end{equation}
Let $f,g\in C^{1}([0,T],H)$
satisfy $f(0)\leq_{K} g(0)$ and $f'(t)- F(t,f(t)) \leq_{K} g'(t)-F(t,g(t))$ for all $t\in [0,T]$. 
Then $f(t)\leq_K g(t) \in K$ for all $t\in [0,T]$.
\end{theorem}
\section{Integration with respect to a vector-valued measure}\label{sec:integr-with-resp}
We summarize some results on vector-valued measures and integration.
The theory goes back to the work of Bartle, Dunford, and Schwartz 
(see, e.g.,~\cite{BDS55}) and Lewis (\cite{Lew70}).
A good overview can be found in~\cite[Chapter 2]{Pan08}. 
As we work in the Hilbert-space setting (in particular, as Hilbert spaces are reflexive), the theory simplifies considerably.\par 

Throughout this section let $(S,\cF)$ be a measurable space, 
let $(H,\langle \cdot, \cdot \rangle_H)$
be a real Hilbert space, and let $\mu\colon \cF\rightarrow H$ be 
an $H$-valued measure.

\begin{definition}
We say that $f\colon S\rightarrow \R$ is \emph{$\mu$-integrable} if the following two conditions are satisfied:
\begin{enumerate}
 \item $f$ is $\langle \mu, h \rangle$-integrable for all $h\in H$ (i.e., $f\colon S \rightarrow \R$ is measurable 
 and $\int_{S} |f|\D|\langle \mu, h \rangle| < \infty$ for all $h\in H$), and 
\item for all $A\in \cF$  there exists an $h_A\in H$ such that for all $h\in H$ 
we have $\langle h_A, h \rangle_H = \int_{A} f \D\langle \mu, h \rangle$.
\end{enumerate}
In this case we denote $h_A$ by $\int_A f \,\D\mu$. In addition, we define 
\begin{equation}
 \mathcal{L}^1(S,\mu) := 
 \{ 
    f \colon S \rightarrow \R \colon 
    f\text{ is $\mu$-integrable}
 \}
\end{equation}
\end{definition}
\begin{example}
If $f$ is a $\cF$-simple function, then $f\in \mathcal{L}^1(S,\mu)$.
\end{example}
The following characterisation is useful (see also~\cite[p.163]{Lew70}): 
\begin{lemma}\label{lem:characint}
We have that $f\in \mathcal{L}^1(S,\mu)$ if and only if 
$f$ is $\langle \mu, h \rangle$-integrable for all $h\in H$.
\end{lemma}
\begin{proof}
Let $(f_n)_{n\in\N}$ be a sequence of simple functions such that $f_n \rightarrow f$ $\mu$-a.s.\
and $|f_n| \leq |f|$ for all $n\in \N$. 
Let $A\in \cF$. Note that the mapping $T\colon H \rightarrow \R$, $T(h) = \int_A f \D\langle \mu, h \rangle $
is linear and that
\begin{align*}
T(h) = 
\lim_{n\rightarrow \infty} \int_A f_n\D\langle \mu, h\rangle = \lim_{n\rightarrow \infty} \langle \int_A f_n\,\D\mu, h\rangle_H,   
\end{align*}
for all $h\in H$ by the dominated convergence theorem. 
It follows from this and the uniform boundedness principle that $\sup_{n\in \N} \|  \int_A f_n \,\D\mu \|_H < \infty$,
whence $T\in H^*$. The Riesz representation theorem thus ensures that there exists an $h_A \in H$
such that $\langle h_A, h \rangle_H = T(h)$ for all $h\in H$.
\end{proof}
\begin{corollary}\label{cor:simpleintcond}
If $f\in \mathcal{L}^1(S,\mu)$
and $g\colon S \rightarrow \R$ is measurable and satisfies $|g|\leq f$ $\mu$-a.s., then $g\in \mathcal{L}^1(S,\mu)$.
In particular, $\mathcal{L}^1(S,\mu)$ contains all bounded measurable $\R$-valued functions on $S$.
\end{corollary}
By~\cite[Corollary 1.4]{Lew70} we have, for any $(E_n)_{n\in \N}$ in $\cF$
converging to $E\in \cF$, that
\begin{equation}\label{eq:meas_continuous} 
  \lim_{n\rightarrow \infty} \mu(E_n) = \mu (E).
\end{equation}
Moreover, the dominated convergence theorem remains valid for $H$-valued measures:
\begin{theorem}[Theorem 2.1.7 in \cite{Pan08}]\label{thm:DCT}
Let $g\in L^1(S,\mu)$, let $f\colon S\rightarrow \R$
be $\mu$-measurable and let $(f_n)_{n\in \MN}$
be a sequence of $\mu$-measurable functions on $S$ satisfying $|f_n(s)|\leq g(s)$ for all $s\in S$, $n\in \MN$,
and $\lim_{n\rightarrow \infty} f_n(s) = f(s)$ for all $s\in S$. Then $f, f_n \in L^1(S,\mu)$, $n\in \MN$,
and 
\begin{equation}
 \lim_{n\rightarrow \infty}
 \left\| 
    \int_S f_n \D\mu
    -
    \int_S f \D\mu
 \right\|_H
 =
 0\,.
\end{equation}

\end{theorem}

Finally, let $K\subset H$ be a self-dual cone and assume that $\mu \colon\cF\rightarrow K$
is a $K$-valued measure. In this case we have $0\leq_K \mu(E)\leq_K \mu(F)$ for all $E,F\in \cF$ 
satisfying $E\subseteq F$, and thus also (by monotonicity of $K$) 
\begin{equation}\label{eq:measuremonotone}
 \|\mu(E)\|_H \leq \| \mu(F)\|_{H}.
\end{equation}
Moreover, as $K$ is self-dual, $\langle \mu, h\rangle$ is a positive measure for all $h\in K$,
whence (again by self-duality) we have 
\begin{equation}\label{eq:int_pos1}
 f\in L^1(S,\mu), f\geq 0 \Rightarrow \int_S f\,\D\mu\in K.
\end{equation}
In particular, if $f\in L^1(S,\mu)$ is positive,
and $E\in \cF$, then 
\begin{equation}\label{eq:int_pos2}
  \int_E f \,\D\mu \leq_{K} \esssup_{s\in E}f(s) \mu(E).
\end{equation}
This combined with the monotonicity of $K$ implies that 
for every every $f\in L^1(S,\mu)$ and every $E\in \cF$ we have 
(by considering $f^+$ and $f^-$ separately) that
\begin{equation}\label{eq:normintest}
 \left\|
    \int_E f \,\D\mu 
 \right\|_{H} 
 \leq \esssup_{s\in E}|f(s)| \| \mu(E) \|_{H}.
\end{equation}
\section{Proof of Proposition \ref{prop:gateaux-diff}}\label{sec:proof-lemma-refl}
To prove Proposition~\ref{prop:gateaux-diff}, we need the following 
consequence of the fundamental theorem of calculus:
\begin{lemma}\label{lem:onesided_FTC}
Let $X,Y$ be Banach spaces, let $F\colon D\subset X \rightarrow Y$, let $x,y\in D$
and assume that the one-sided derivative of $F$ in $z$ exists in the direction $y-x$ for all $z\in \{x+s(y-x)\colon s\in [0,1]\}$ and that the mapping 
\begin{equation} [0,1]\ni s \mapsto \D_{+}F(x+s(y-x))(y-x) \in Y\end{equation} is continuous. 
Then $F(y) - F(x) = \int_{0}^{1} \D_{+}F(x+s(y-x))(y-x)\,\D s$.
\end{lemma}

\begin{proof}
The continuity of $[0,1]\ni s \mapsto \D_{+}F(x+s(y-x))(y-x) \in Y$ and the fundamental theorem of calculus
imply that the right derivative of the mapping 
$[0,1]\ni t \mapsto \left( F(x+t(y-x)) - F(x) - \int_{0}^{t} \D_{+}F(x+s(y-x))(y-x)\,\D s \right) \in Y$ 
equals zero. As any function with right derivative equal to zero is constant, this leads to the desired assertion.
\end{proof}

\begin{proof}[Proof of Proposition~\ref{prop:gateaux-diff}]
Note that in order to prove that the \emph{second} directional derivative in
$0$ of a mapping exists, we need that its 
\emph{first} directional derivative exists in $u\in\cH^+$ for all $u\in \cH^+$ sufficiently small. 
Hence, we begin by proving that the first derivative of $u\mapsto \psi(t,u)$ exists in $u$ in the direction $v$
for all $u,v\in \cH^+$ and all $t\in [0,\infty)$. To this end we fix $u,v\in \cH^+$.\par 
Recall the definition of the operators $\D R(u)\in \cL(\cH)$ and $\D^2 R(u) \in \cL^{(2)}(\cH\times \cH, \cH)$ from~\eqref{eq:def_dR} and \eqref{eq:def_d2R}. Define
the operator $C_{\theta}(t)\in\cL(\cH)$, $\theta, t\in [0,\infty)$, by 
\begin{equation}\label{eq:def_Ctheta}
 C_{\theta}(t)w 
 = 
 \int_{0}^{1} 
    \D R\left(
        \psi(t,u) 
        + 
        s(\psi(t,u+\theta v) - \psi(t,u))
    \right)
    w
 \,\D s 
\end{equation}
(note that the integral is well-defined as the integrand is continuous in $s$ 
by~\eqref{eq:dR_lipschitz} and bounded by~\eqref{eq:expgrowth_sol_riccati} and~\eqref{eq:dR_bdd}). 
Lemma~\ref{lem:onesided_FTC},~\eqref{eq:R-one-sided-derivative}, 
the fact that $(1-s) \psi(t,u) + s \psi(t,u+\theta v) \in \cH^+$ for all $s\in [0,1]$, $t\in [0,\infty)$,
and the fact that $\psi(t,u+\theta v) \geq_{\cH^+} \psi(t,u)$ for all $t\in [0,\infty)$ by~\eqref{eq:monotone_sol_riccati} imply that 
\begin{equation*}
 C_{\theta}(t)(\psi(t,u+\theta v)- \psi(t,u)) = R(\psi(t,u+\theta v))-R(\psi(t,u)),\quad 
 \theta, t\in [0,\infty).
\end{equation*}
This and~\eqref{eq:riccati-conc} imply
\begin{equation*}
 \frac{\partial}{\partial t}
 (\psi(t,u+\theta v) - \psi(t,u))
 =
 C_{\theta}(t)
 (\psi(t,u+\theta v) - \psi(t,u)),\quad \theta, t\in [0,\infty).
\end{equation*}
It follows that
\begin{equation*}
 \psi(t,u+\theta v) - \psi(t,u)
 =
 \theta \exp\left(\int_{0}^{t} C_{\theta}(s)\,\D s\right)v,\quad \theta, t\in [0,\infty)
\end{equation*}
(note that $\int_{0}^{t} C_{\theta}(s) \,\D s$ is well-defined in $\cL(\cH)$
as the $\cL(\cH)$-valued integrand is continuous in $s$ 
by~\eqref{eq:dR_lipschitz} and bounded due to~\eqref{eq:dR_bdd}).
This implies that for all $\theta\in (0,\infty)$ we have  
\begin{equation}\label{eq:diff_est_phi}
\begin{aligned}
& \left\|
 \frac{\psi(t,u+\theta v) - \psi(t,u)}
 {\theta}
 -
 \exp\left(\int_{0}^{t} C_{0}(s)\,\D s\right)v
\right\|
 \\ & =
 \left\|
 \left(
 \exp\left(\int_{0}^{t} C_{\theta}(s)\,\D s\right)
 - 
 \exp\left(\int_{0}^{t} C_{0}(s)\,\D s\right)
 \right)v
 \right\|.
\end{aligned}
\end{equation}
Using the identity $\| \E^{A} - \E^{B} \|_{\cL(\cH)} \leq \| A-B\|_{\cL(\cH)} \E^{\| A \|_{\cL(\cH)} \vee \| B \|_{\cL(\cH)}}$,
$A,B\in \cL(\cH)$, we obtain from~\eqref{eq:def_Ctheta},~\eqref{eq:diff_est_phi},~\eqref{eq:expgrowth_sol_riccati},~\eqref{eq:continuity_sol_riccati}, and~\eqref{eq:dR_lipschitz}
that the one-sided derivative $\D_{+}\psi(t,u)(v)$ exists.
Moreover, the fact that $C_0(t)v = \D R(\psi(t,u))v$ implies that $t\mapsto \D_{+}\psi(t,u)(v)$ is the solution to the following ODE
  \begin{align}\label{eq:dif-eq-phi-1-everywhere}
    \frac{\partial}{\partial t} \D_{+}\psi(t,u)(v)= \D R(\psi(t,u))\big(\D_{+}\psi(t,u)(v)\big),
    \quad t\geq 0;
    \quad \D_{+}\psi(0,u)(v)=v.
  \end{align} 
This together with the quasi-monotonicity of $\D R(\psi(t,u))$ (see Lemma~\ref{lem:R-one-sided-dif})
and Theorem~\ref{thm:comparison-theorem} implies that $\D_{+}\psi(t,u)(v) \in \cH^+$.
Regarding the derivative of $\phi$, note that estimates analogous to~\eqref{eq:dR_bdd} and~\eqref{eq:dR_lipschitz} hold for $\D F$, which, in combination with the fact that $\D_{+}\psi(t,u)(v) \in \cH^+$,~\eqref{eq:riccati-conc},~\eqref{eq:F-one-sided-derivative}, and Lemma~\ref{lem:onesided_FTC} implies that
\begin{align*}
 & \frac{\phi(t,u+\theta v) - \phi(t,u)}
 {\theta}
 \\ & =
 \int_{0}^{t}
    \int_{0}^{1}
        \D F(\psi(s,u)+r(\psi(s,u+\theta v)-\psi(s,u))) 
    \,\D r
    \frac{\psi(s,u+\theta v) - \psi(s,u)}
    {\theta}
 \,\D s
\end{align*}
for all $\theta \in (0,\infty), t\in [0,\infty)$. This in combination 
with~\eqref{eq:continuity_sol_riccati} and~\eqref{eq:monotone_sol_riccati}
implies that the dominated convergence theorem can be applied to obtain that $\D_+ \phi(t,u)$ exists 
for all $t$ and satisfies
\begin{align}\label{eq:dif-eq-psi-1-everywhere}
 \frac{\partial}{\partial t} \D_{+}\phi(t,u)(v)= \D F(\psi(t,u))\big(\D_{+}\psi(t,u)(v)\big),
    \quad t\geq 0;
    \quad \D_{+}\phi(0,u)(v)=0.
  \end{align} 
This proves in particular that $u\mapsto (\phi(t,u),\psi(t,u))$ is differentiable
in $0$ in the direction $v\in \cH^+$ for all $v\in \cH^+$ and that the corresponding derivatives 
solve the ODEs~\eqref{eq:dif-eq-phi-1} and~\eqref{eq:dif-eq-psi-1}.\par 
We now turn to the second derivative in $0$. To this end, fix $v,w\in \cH^+$ and observe that 
Lemma~\ref{lem:onesided_FTC}, the boundedness and continuity of $\D^2 R$ (see Lemma~\ref{lem:R-one-sided-dif}),~\eqref{eq:R-one-sided-dif-second} and the fact that $\psi(t,\theta v),\D_+\psi(t,\theta v)\in \cH^+$ for all $\theta\in [0,\infty)$
imply that
\begin{align*}
 \frac{\partial}{\partial t}
 \left(
    \D_+\psi(t,\theta v)(w) - \D_+\psi(t,0)(w)
 \right)
 & =
 \int_{0}^{1} \D^2 R(s\psi(t,\theta v))(\D_+\psi(t,\theta v)(w),\psi(t,\theta v)) \,\D s 
 \\ &\quad +
 \D R(0) \left(
    \D_+\psi(t,\theta v)(w) - \D_+\psi(t,0)(w)
 \right)
\end{align*}
for all $\theta \in [0,\infty), t\in [0,\infty)$.
As $\D_+\psi(0,\theta v)(w) - \D_+\psi(0,0)(w)=0$ this implies
\begin{align}
& \frac{
    \D_+\psi(t,\theta v)(w) - \D_+\psi(t,0)(w)
 }{\theta}
\notag \\
& =
 \int_{0}^{t}
    \E^{(t-r)\D R(0)}
    \int_{0}^{1} 
        \D^2 R(s\psi(r,\theta v))
        \left(
            \D_+\psi(r,\theta v)(w) ,
            \frac{\psi(r,\theta v)}{\theta}
        \right)
    \,\D s
 \,\D r \label{eq:diff_eq_d2R}
\end{align}
for all $\theta \in (0,\infty), t\in [0,\infty)$.
Note that~\eqref{eq:continuity_sol_riccati},~\eqref{eq:dR_lipschitz}, and~\eqref{eq:dif-eq-phi-1-everywhere} imply that $\lim_{\theta \to 0+}\D_+\psi(t,\theta v)(w)=\D_{+}\psi(t,0)(w)$. Moreover, we have already established that  
$\lim_{\theta \to 0+} \frac{\psi(t,\theta v)}{\theta} = \D_+ \psi(t,0)(v)$. Combining 
these observations with~\eqref{eq:expgrowth_sol_riccati},~\eqref{eq:d2R_bdd}, and~\eqref{eq:diff_eq_d2R}
implies that $\D^2_+\psi(t,0)(v,w)$ exists and that $\D^2_+\psi(t,0)(v,w)$ satisfies~\eqref{eq:dif-eq-phi-2}. We leave it to the reader to now verify that also $\D^2_+ \phi(t,u)(v,w)$ exists and that $\D^2_+ \phi(t,u)(v,w)$ satisfies~\eqref{eq:dif-eq-psi-2}.
\end{proof}


\end{document}